\documentclass[a4paper,11pt]{amsart}
\usepackage{amsmath,inputenc,euscript,amssymb,geometry}
\usepackage{graphicx}
\usepackage{amssymb}
\usepackage{latexsym}
\usepackage{amssymb,amsbsy,amsmath,amsfonts,amssymb,amscd,color}
\usepackage{hyperref}

\newtheorem{lemma}{Lemma}[section]
\newtheorem{theorem}[lemma]{Theorem}
\newtheorem{prop}[lemma]{Proposition}

\newtheorem{rem}[lemma]{Remark}
\newtheorem{remark}[lemma]{Remark}

\begin{document}
\title[]{Evolution of polygonal lines by the binormal flow}
\author[V. Banica]{Valeria Banica}
\address[V. Banica]{Laboratoire Jacques-Louis Lions (UMR 7598)
B.C. 187, 4 place Jussieu
75005 Paris, France, Valeria.Banica@ljll.math.upmc.fr} 

\author[L. Vega]{Luis Vega}
\address[L. Vega]{Departamento de Matem\'aticas, Universidad del Pais Vasco, Aptdo. 644, 48080 Bilbao, Spain, luis.vega@ehu.es} 
\date\today
\maketitle
\begin{abstract} The aim of this paper is threefold. First we display solutions of the cubic nonlinear Schr\"odinger equation on $\mathbb R$ in link with initial data a sum of Dirac masses. Secondly we show a Talbot effect for the same equation. Finally we prove the existence of a unique solution of the binormal flow with datum a polygonal line. This equation is used as a model for the vortex filaments dynamics in 3-D fluids and superfluids. We also construct solutions of the binormal flow that present an intermittency phenomena. Finally, the solution we construct for the binormal flow is continued for negative times, yielding a geometric way to approach the continuation after blow-up for the 1-D cubic nonlinear Schr\"odinger equation. 
\end{abstract}

\section{Introduction}
We first present the binormal flow framework and the obtained results. Then in \S\ref{ssectNLS} we describe the 1-D cubic nonlinear Schr\"odinger equation results. 
\subsection{Evolution of polygonal lines through the binormal flow and intermittency}\label{chBF} 
Vortex filaments in 3-D fluids appear when vorticity is large and concentrated in a thin tube around a curve in $\mathbb R^3$. The binormal (curvature) flow, that we refer hereafter as BF, is the classical model for one vortex filament dynamics. It was derived by Da Rios 1906 in his PhD advised by Levi-Civita by using a truncated Biot-Savart law and a renormalization in time (\cite{DaR}).  The evolution of a $\mathbb{R}^3$-curve $\chi(t)$ parametrized by arclength $x$   by the binormal flow is 
\begin{equation}\label{bf}
\chi_t=\chi_x\wedge\chi_{xx}. 
\end{equation}
Keeping in mind the Frenet's system for the frames of 3-D curves composed by tangent, normal, and binormal vectors  $(T,n,b)$ 
$$\left(\begin{array}{c}
T\\n\\b
\end{array}\right)_x=
\left(\begin{array}{ccc}
0 & c & 0 \\ -c & 0 & \tau \\ 0 &  -\tau & 0 
\end{array}\right)
\left(\begin{array}{c}
T\\n\\b
\end{array}\right),$$
where $c,\tau$ are the curvature and torsion, the binormal flow can be rewritten as
$$\chi_t=c\,b.$$
BF was also derived as formal asymptotics in \cite{ArHa}, and in \cite{CaTi} by using  the technique of matched asymptotics in the Navier-Stokes equations (i.e. to balance the cross-section of the tube with the Reynolds number).  In the recent paper \cite{JeSe}, and still under some hypothesis on the persistence of concentration of vorticity in the tube, BF is rigorously derived; moreover the considered curves are not necessarily smooth. This is based on the existence of a correspondence between the two Hamilton-Poisson structures that give rise to Euler and to  BF. 

Existence results were given for curves with curvature and torsion in Sobolev spaces of high order (\cite{Ha},\cite{NiTa},\cite{FuMi},\cite{Ko}), and more generally existence results for currents in the framework of a weak formulation of the binormal flow (\cite{JeSm2}). Recently, the Cauchy problem was shown to be well-posed for curves with a corner and curvature in weighted space (\cite{BV4}). 

An important feature of BF is that the tangent vector of a solution $\chi(t)$ solves the Schr\"odinger map onto $\mathbb S^2$:
$$T_t=T\wedge T_{xx}.$$ Furthermore, Hasimoto remarked in \cite{Ha} that the function, that he calls the filament function,
$u(t,x)=c(t,x)e^{i\int_0^x\tau(t,s)ds}$, satisfies a focusing 1-D cubic nonlinear Schr\"odinger equation (NLS)
\footnote{The defocusing 1-D cubic Schr\"odinger equation is achieved if the target of the Schr\"odinger map equation is the hyperbolic plane $\mathbb H^2$ instead of the sphere $\mathbb S^2$.}. 
Hasimoto's transform can be viewed as an inverse Madelung transform sending Gross-Pitaesvskii equation to compressible Euler equation with quantum pressure. It is known that in order to avoid issues related to vanishing curvature, Bishop parallel frames (\cite{Bi},\cite{Ko}) can be used as explained in \S\ref{sectconstr}.

Several examples of evolutions of curves through the binormal flow were given finding first particular solutions of the 1-D cubic NLS and then solving the corresponding Frenet equations.  Some of these example are consistent at the qualitative level with classical vortex filament dynamics as the line, the ring, the helix and travelling wave type vortices. A special case are the self-similar solutions of the binormal flow. They are constructed from the solutions 
 \begin{equation}\label{sssol}u_\alpha(t,x)=\alpha\frac{e^{i\frac{x^2}{4t}}}{\sqrt{4\pi it}}=\alpha e^{it\Delta}\delta_0(x)
 \end{equation}
of the 1-D cubic NLS  equation, renormalized in a sense specified in \S\ref{ssectNLS}, with a Dirac mass $\alpha\delta_0$ at initial time. These BF solutions are of the type $\chi(t,x)=\sqrt{t}G(\frac x{\sqrt{t}})$, and form a a 1-parameter family $\{\chi_{\alpha}, \alpha\geq 0\}$, with $\chi_\alpha(t)$  characterized by its curvature $c_\alpha(t,x)=\frac \alpha{\sqrt{t}}$ and its torsion $\tau_\alpha(t,x)=\frac x{2t}$. These solutions were known and used for quite a while in the 80's (\cite{LD},\cite{LRT},\cite{Bu},\cite{Sc}). The existence of a trace at time $t=0$ was proved rigorously in \cite{GRV}, and in particular it was shown that $\chi_\alpha(0)$ is a broken line with one corner having an angle $\theta$ satisfying 
 \begin{equation}\label{angless}\sin\left(\frac{\theta}{2}\right)=e^{-\pi\frac{\alpha^2}{2}}.\end{equation}
In particular the Dirac mass at the NLS level corresponds to the formation of a corner on the curve, but the trace $\alpha\delta_0$ of the filament function is not the filament function $\theta\delta_0$ of $\chi_\alpha(0)$.  This turns out to have relevant consequences regarding the lack of continuity of some norms at the time when the corner is created. In \cite{BVnote} it is proved the $\| \widehat{T_x}(\,\cdot, t)\|_\infty$ is discontinuous at that time. The same proof works if instead of this norm it is used the following one
$$\sup_j\int_{4\pi j}^{4\pi (j+1)}| \widehat{T_x}( x, t)|^2\,dx,$$ 
that fits better within the framework of Theorem \ref{thdiracs}, because due to the Frenet equations $T_x$ it is at the same level of regularity as the corresponding filament functions that solve NLS.

We shall now turn our attention precisely to the evolution of curves that can generate corners in finite time. The case of the formation and instantaneous disappearance of one corner is now well understood thanks to the characterization of the family the self-similar solutions, and the study in \cite{BV4} of the evolution of non-closed curves with one corner and with curvature in weighted $L^2$ based spaces.  On the other hand, a planar regular polygon with $M$ sides is expected to evolve through the binormal flow to skew polygons with $Mq$ sides at times  $t_{p,q}=\frac p{2\pi q}$ for odd $q$ , see the numerical simulations in \cite{GrDe},\cite{JeSm2}, and \cite{DHV} where the integration of the Frenet equations at the rational times $t_{p,q}$ is also done.

In the present paper we place ourselves in the framework of initial data being polygonal lines. The results presented are an important step forward to fill the gap between the case of one corner and the much more delicate issue of closed polygons.
 
 \begin{theorem}\label{brokenline}{\bf{(Evolution of polygonal lines through the binormal flow)}} Let $\chi_0$ be  an arclength parametrized polygonal line with corners located at $x\in\mathbb Z$, with the sequence of angles $\theta_n\in(0,\pi)$ such that the sequence defined by (cf. \eqref{angless})
 \begin{equation}\label{angle}\sqrt{-\frac {2}{\pi}\log\left(\sin\left(\frac{\theta_n}{2}\right)\right)},\end{equation}
belongs to $l^{2,3}$. Then there exists $\chi(t)$, smooth solution of the binormal flow \eqref{bf} on $t\neq0$ and solution of \eqref{bf} in the weak sense on $\mathbb R$, with 
$$|\chi(t,x)-\chi_0(x)|\leq C\sqrt{t},\quad \forall x\in\mathbb R, |t|\leq 1.$$
 \end{theorem}
 
 \begin{remark}
 Under suitable conditions on the initial data $\chi_0$, the evolution can have an intermittent behaviour:  Proposition \ref{proptalbotnl} insures that at times $t_{p,q}=\frac 1{2\pi}\frac pq$ the curvature of $\chi(t)$ displays concentrations near the locations $x$ such that $x\in\frac 1q\mathbb Z$, and $\chi(t)$ is almost a straight segment in between. 
 \end{remark}
 


The proof goes as follows. In view of \eqref{angless} and Hasimoto's transform we consider an appropriate 1-D cubic NLS equation with initial data 
$$\sum_{k\in\mathbb Z} \alpha_k\delta_k,$$
with $\alpha_k$ complex numbers defined in a precise way from the curvature and torsion angles of $\chi_0$. 
Theorem \ref{thdiracs} gives us a solution $u(t)$ on $t>0$. From this smooth solution on $]0,\infty[$ we construct a smooth solution $\chi(t)$ of the binormal flow on $]0,\infty[$, that we prove it has a limit $\chi(0)$ at $t=0$. Then the goal is to show that modulo a translation and a rotation $\chi(0)$ is $\chi_0$. This is done in several steps. First we show that the tangent vector has a limit at $t=0$. Secondly we show that this limit is piecewise constant, so $\chi(0)$ is a segment for $x\in]n,n+1[,\forall n\in\mathbb Z$. Then we prove, by analyzing the frame of the curve through paths of self-similar variables, that $\chi(0)$ presents corners at the same locations as $\chi_0$, of same angles as $\chi_0$. We recover the torsion angles of $\chi_0$ by using also a similar analysis for modulated normal vectors $\tilde N(t,x)=e^{i\sum_{j\neq x}|\alpha_j|^2\log\frac{x-j}{\sqrt{t}}}N(t,x).$  
Therefore  we recover $\chi_0$ modulo a translation and a rotation. This translation and rotation applied to $\chi(t)$ give us the desired solution of the binormal flow for $t>0$ with limit $\chi_0$ at $t=0$. Uniqueness holds in the class of curves having filament functions of type \eqref{ansatzfin}. Using the above recipe to construct the evolution of a polygonal line for $t>0$ we can extend $\chi(t)$ to negative times by using the time reversibility of the equation.

\subsection{The cubic NLS on $\mathbb R$ with initial data given by several Dirac masses.}\label{ssectNLS} We consider the cubic nonlinear Schr\"odinger equation on $\mathbb R$
\begin{equation}\label{cubic}i\partial_t u+\Delta u \pm \frac 12 |u|^{2}u=0.\end{equation}

We first recall the known local well-posedness results, starting with what is known in the framework of Sobolev spaces. The equation is well-posed in $H^{s}$, for any $s\geq 0$ (\cite{GiVe},\cite{CaWe}). On the other hand, for $s<0$ the Cauchy problem is ill-posed: in \cite{KPV} uniqueness was proved to be lost by using the Galilean transformation, and in \cite{ChCoTa} norm-inflation phenomena were displayed. We note that the threshold obtained with respect of the scaling invariance $\lambda u(\lambda^2t,\lambda x)$ is $\dot H^{-\frac 12}$. For $s\leq -\frac 12$ the presence of norm inflating phenomena with loss of regularity was pointed out in \cite{CaKa},\cite{Ki}, and also norm inflating around any data was proved in \cite{Oh}. Finally growth control of Sobolev norms of Schwartz solutions for $-\frac 12<s<0$ on the line or the circle was shown in \cite{KiViZh} and \cite{KoTa}.

On the other hand well-posedness holds for data with Fourier transform in $L^p$ spaces, $p<+\infty$ (\cite{VaVe},\cite{Gr},\cite{Ch}). A natural choice would be to consider initial data with Fourier transform in $L^\infty$, as this space $\mathcal F(L^\infty)$ it is also invariant under rescaling. 

We shall now focus on the case of initial data of Dirac mass type. Note that the Dirac mass is borderline for $\dot H^{-\frac 12}$ and that it belongs to $\mathcal F(L^\infty)$. 
For an initial datum given by one Dirac mass, $u(0)=\alpha\delta_0$, the equation is ill-posed. More precisely, it is showed in \cite{KPV} by using the Galilean invariance that if there exists a unique solution, it should be for positive times 
$$\alpha\frac{e^{\mp i\frac{|\alpha|^2}{4\pi}\log \sqrt{t}+i\frac{x^2}{4t}}}{\sqrt{4\pi it}},$$
 and then the initial datum is not recovered. We note here that this issue can be avoided by a simple change of phase that leads to the equation
$$\left\{\begin{array}{c}i\partial_t u+\Delta u \pm\frac 12 \left(|u|^{2}-A(t)\right)u=0,\\ u(0)=\alpha\delta_0,\end{array}\right.$$
with $A(t)=\frac{\alpha^2}{4\pi t}$. With this choice the equation has as a solution precisely the fundamental solution of the linear equation $u_\alpha(t,x)$ introduced in \eqref{sssol}. Adding a real potential $A(t)$ is a very natural geometric normalization as the BF solution constructed from a NLS solution $u(t,x)$ is the same as the one constructed from $e^{i\phi(t)}u(t,x)$, see \S\ref{sectconstr}. This type of Wick renormalization has been used in the periodic setting in previous works as in \cite{Bo96},\cite{Ch},\cite{OhSu} and \cite{OhWa}, although the motivation in these cases came just from the need of avoiding some resonant terms that become infinite.

However, even with this geometric renormalization the problem is still ill-posed, in the sense that small regular perturbations of $u_\alpha(t)$ at time $t=1$ were proved in \cite{BV2} to behave near $t=0$ as $u_\alpha(t)+e^{i\log t} f(x)$ for some $f\in H^1$. Therefore there is a loss of phase as $t$ goes to zero. 

This loss of phase is a usual phenomena in the setting of the nonlinear Schr\"odinger equations when singularities are formed, and it is of course a consequence of the gauge invariance of the equation. How to continue the solution after the singularity has been formed is therefore an important issue that appears recurrently in the literature, see for example \cite{Me92},\cite{Me92bis},\cite{BoWa},\cite{MeRaSz}. 

In \cite{BV4} we found a natural geometric way to continue the BF solution after the singularity, in the shape of a corner, is created. As BF is time reversible, to uniquely continue a solution for negative times requires to get a curve trace $\chi(0)$ at $t=0$ and to construct a unique solution for positive times, having as limit at $t=0$ the inverse oriented curve $\chi(0,-s)$. 
Note that using just continuity arguments and the characterization result of the self-similar solutions that was proved in \cite{GRV} one can construct in an artificial way the continuation of a self-similar solution. 
A more delicate issue is how to determine the curve trace and its Frenet frame at time $t=0$ for small regular perturbations of BF self-similar solutions at some positive time, and we based our analysis in \cite{BV4} on the characterization result of the self-similar solutions that was proved in \cite{GRV}; in particular the small regular perturbations of BF self-similar solutions at some positive time do not break the self-similar symmetry of the singularity created at $t=0$. 



In Theorem \ref{brokenline} we prove that this procedure can be extended, not without difficulties, to the case of a polygonal line, that can be viewed as a rough perturbation of broken line with one corner. There is no need for the line to be planar and infinitely many corners are permitted. In this case new problems concerning the phase loss appear at the NLS and frame level and again the characterization of the self-similar solutions plays a crucial role. 

For these reasons in this article we consider as initial data a combination of Dirac masses,
\begin{equation}\label{id}u(0)=\sum_{k\in\mathbb Z}\alpha_k\delta_k,\end{equation}
with coefficients in weighted summation spaces : 
$$ \|\{\alpha_k\}\|_{l^{p,s}}<\infty,$$
where
$$ \|\{\alpha_k\}\|_{l^{p,s}}^p:=\sum_{k\in\mathbb Z}(1+|k|)^{ps}|\alpha_k|^p.$$
This choice of initial data has its own interest from the point of view of the Schr\"odinger equation, because as far as we know and for the cubic nonlinearity in one dimension the only results at the critical level of regularity are the ones in \cite{BV4} mentioned above and that deals with just one Dirac mass. The case of a periodic array of Dirac deltas of the same precise amplitude, was studied in \cite{DHV} where a candidate for a solution is proposed.

The case of a combination of Dirac masses as initial data for the Schr\"odinger equation  $|u|^{p-1}u$ with subcritical nonlinearity $p<3$ was   considered in \cite{K}. It was showed that it admits a unique solution, of the form 
\begin{equation}\label{ansatz}u(t,x)=\sum_{k\in\mathbb Z}A_k(t)e^{it\Delta}\delta_k(x),\end{equation}
where $\{A_k\}\in\mathcal C([0,T];l^{2,1})\cap C^1(]0,T];l^{2,1})$.
As the nonlinear power approaches the critical cubic power, things look more singular. In this paper we prove that the same type of ansatz is valid for a naturally renormalized cubic  equation. 

Let us notice that the initial data \eqref{id} has the property
\begin{equation}\label{Fourper}\widehat {u(0)}(\xi)=\sum_{k\in\mathbb Z}\alpha_ke^{-ik\xi},\end{equation}
and in particular $\widehat {u(0)}$ is $2\pi-$periodic. 
Moreover, the condition $\{\alpha_k\}\in l^{2,s}$ translates into $\widehat {u(0)}\in H^s(0,2\pi)$. Conversely, every $2\pi-$periodic function can be decomposed as in \eqref{Fourper} and so it represents the Fourier transform on $\mathbb R$ of a combination of Dirac masses as \eqref{ansatz}.
We denote
$$H^s_{pF}:=\{u\in\mathcal S'(\mathbb R),\,\, \hat u(\xi+2\pi)=\hat u(\xi), \hat u\in H^s(0,2\pi)\}\subset \{u\in\mathcal S'(\mathbb R),\,\,\{\|\hat u\|_{H^s(2\pi j,2\pi(j+1))}\}_j\in l^\infty\},$$
and $$\|u\|_{H^s_{pF}}=\|\hat u\|_{H^s(0,2\pi)}.$$
Our first result concerns the existence of solutions for initial data in $H^s_{pF}$.

\begin{theorem}\label{thdiracs}{\bf{(Solutions of 1-D cubic NLS linked to several Diracs masses as initial data)}} Let $s>\frac 12, 0<\gamma<1$ and $\{\alpha_k\}\in l^{2,s}$. We consider the 1-D cubic NLS equation:
\begin{equation}\label{cubicmod}\begin{array}{c}i\partial_t u +\Delta u\pm\frac 12\left(|u|^2-\frac{M}{2\pi t}\right)u=0,
\end{array}
\end{equation}
with $M=\sum_{k\in\mathbb Z}|\alpha_k|^2$. There exists $T>0$ and a unique solution on $(0,T)$ of the form 
\begin{equation}\label{ansatzfin}u(t,x)=\sum_{k\in\mathbb Z}e^{\mp i\frac{|\alpha_k|^2}{4\pi}\log \sqrt{t}}(\alpha_k+R_k(t))e^{it\Delta}\delta_k(x),\end{equation}
 with 
 \begin{equation}\label{decayansatzcubic}\sup_{0<t<T}t^{-\gamma}\|\{R_k(t)\}\|_{l^{2,s}}+t\,\|\{\partial_t R_k(t)\}\|_{l^{2,s}}<C.\end{equation}
Moreover, considering as initial data a finite sum of $N$ Dirac masses 
$$u(0)=\sum_{k\in\mathbb Z}\alpha_k\delta_k,$$
with coefficients of equal modulus 
\begin{equation}\label{01}
|\alpha_k|=a,
\end{equation}
 and equation \eqref{cubicmod} renormalized with $M=(N-\frac 12)a^2$, we have a unique solution
$$u(t)=e^{it\Delta}u(0)\pm ie^{it\Delta}\int_0^te^{-i\tau\Delta}\left(\left(|u(\tau)|^2-\frac{M}{2\pi \tau}\right)u(\tau)\right)\,\frac{d\tau}2,$$
such that $\widehat{e^{-it\Delta}u(t)}\in\mathcal C^1((-T,T),H^s(0,2\pi))$ with 
$$\|e^{-it\Delta}u(t)-u(0)\|_{H^s_{pF}}\leq C t^\gamma, \quad\forall t\in(-T,T).$$
Moreover, if $s\geq 1$ then the solution is global in time. \par

\end{theorem}


\begin{remark} Note that any $\alpha_j $ such that \eqref{01} does not hold will imply that the corresponding initial value problem is ill posed, similarly at what was proved in \cite{KPV} and \cite{BV4} in the case of just one Dirac mass and that we mentioned above.
\end{remark}

\begin{remark}
It is worth noting that performing the (reversible) pseudo-conformal transform to the solution $u$ of  \eqref{cubicmod}
$$u(t,x)=\frac{e^{i\frac{x^2}{4t}}}{\sqrt{4\pi it}}\overline{v}(\frac 1t,\frac xt),\qquad t>0$$
we obtain a solution $v$ of
\begin{equation}\label{cubicmodPC}i\partial_t v +\Delta v\pm\frac 1{8\pi t}\left(|v|^2-2M\right)v=0.
\end{equation}
This was the procedure we used in \cite{BV4}.

To impose the ansatz \eqref{ansatz} on $u$ is equivalent to
\begin{equation}\label{ansatzPC}v(t,x)=\sum_{k\in\mathbb Z}\overline{A_k}(\frac 1t)e^{-i\frac{tk^2}{4}+i\frac {xk}2}.\end{equation}
Therefore after pseudo-conformal transform our problem reduces to solve \eqref{cubicmodPC} in the periodic setting with period  $[0,4\pi]$. Note that from \eqref{ansatz} we have that $|\widehat{u(t)}(\xi)|$ is $2\pi$ periodic. 
\end{remark}

The proof of the theorem goes as follows. Plugging the general ansatz \eqref{ansatz} 
into equation \eqref{cubicmod} leads to a discrete system on $\{A_k(t)\}$, by  using the fact that for fixed $t$ the family $e^{it\Delta}\delta_k(x)=\frac{e^{i\frac{(x-k)^2}{4t}}}{\sqrt{4\pi it}}$ is an orthonormal family of $L^2(0,4\pi t)$. We solve the discrete system on $\{A_k(t)\}$ by a fixed point argument with $R_k(t)=e^{-i\frac{|\alpha_k|^2}{4\pi}\log \sqrt{t}}A_k(t)-\alpha_k$ satisfying \eqref{decayansatzcubic}. In the case of initial data a finite sum of $N$ Dirac masses with coefficients of equal modulus and equation \eqref{cubicmod} renormalized with $M=(N-\frac 12)a^2$, we are led to solve the same fixed point for $R_k(t)=A_k(t)-\alpha_k$.

\begin{remark}\label{l^2}The resonant part of the discrete system of $\{A_k(t)\}$ is
$$i\partial a_k(t)=\frac1{8\pi t}a_k(t)(2\sum_j|a_j(t)|^2-|a_k(t)|^2-2M).$$
It is a non-autonomous singular time-dependent coefficient version of the resonant system of standard 1-D cubic NLS. Indeed, usually for questions concerning the long-time behavior of cubic NLS, one introduces
$$v(t)=e^{-it\Delta}u(t).$$
In the 1-D periodic case the Fourier coefficients of $v(t)$ satisfy the system 
\begin{equation*}\label{rescubic} i\partial_t v_k(t)=\sum_{k-j_1+j_2-j_3=0}e^{-it(k^2-j_1^2+j_2^2-j_3^3)}v_{j_1}(t)\overline{v_{j_2}(t)}v_{j_3}(t),\end{equation*}
so that the resonant system is:
$$i\partial a_k(t)=a_k(t)(2\sum_j|a_j(t)|^2-|a_k(t)|^2).$$
Of course, for  1-D periodic NLS with data in $H^s, s>\frac 12$ (that corresponds to $\{v_n(0)\}\in l^{2,s}\subset l^1$) there is no issue for obtaining directly the local existence.
\end{remark}

\begin{remark}The regularity of $\{\alpha_j\}$ might be weakened to $l^p$ spaces only ($p<\infty$), see Remark \ref{remlow}. It is evident from \eqref{cubicmodPC} that formally
\begin{equation}\label{02}\partial_t\sum_j|A_j(t)|^2=0,
\end{equation}
and therefore the $l^2$ norm
is preserved
\footnote {equivalently $\int _0^{4\pi}|v(t,x)|^2\,dx=constant$}. As a matter of fact this says that the selfsimilar solutions have finite mass for the 1-D cubic NLS when the mass is appropriately defined. This has nothing to do with the complete integrability of the system because still works in the subcritical cases studied in \cite{K}. 

Note that to solve \eqref{cubicmodPC} for $t\geq T_0>0$ is quite straightforward making use of the available Strichartz estimates in the periodic setting -see \cite{Bo93} and also \cite{MoVe} for a slight modification. However, these methods do not give the behavior of the solution $v$ when time approaches infinity which is absolutely crucial for proving Theorem \ref{brokenline}. As a consequence we are led to make a more refined analysis. In view of Theorem \ref{brokenline} we consider weighted $l^{2,s}$ spaces; this in particular will allow us to rigorusly prove that \eqref{02} holds.
\end{remark}


The paper is structured as follows. In the next section we prove Theorem \ref{thdiracs}, and also the extension Theorem \ref{thdiracs3} concerning some cases of Dirac masses not necessary located at integer numbers. Section \ref{sectTalbot} contains the proof of a Talbot effect for some solutions given by Theorem \ref{thdiracs}. In the last section we prove Theorem \ref{brokenline}.

\medskip 
{\bf{Acknowledgements:}} Both authors were partially supported by the by an ERCEA Advanced Grant 2014 669689 - HADE. The second author was partially by the MEIC project MTM2014-53850-P and MEIC Severo Ochoa excellence accreditation SEV-2013-0323..\\



\section{The 1-D cubic NLS with initial data given by several Dirac masses}
In this section we give the proof of Theorem \ref{thdiracs}.
\subsection{The fixed point framework}
We denote $\mathcal N(u)=\frac{|u|^2u}2$. By plugging the ansatz \eqref{ansatz} into equation \eqref{cubicmod} we get
\begin{equation}\label{syst}\sum_{k\in\mathbb Z}i\partial_t A_k(t)e^{it\Delta}\delta_k=\mathcal N(u)-\frac{M}{4\pi t}u=\mathcal N(\sum_{j\in\mathbb Z}A_j(t)e^{it\Delta}\delta_j)-\frac{M}{4\pi t}(\sum_{k\in\mathbb Z}A_k(t)e^{it\Delta}\delta_k).\end{equation}
We have chosen here for simplicity the sign $-$ in \eqref{cubicmod}; the sign $+$ can be treated the same. \\
The family $e^{it\Delta}\delta_k(x)=\frac{e^{i\frac{(x-k)^2}{4t}}}{\sqrt{4\pi it}}$ is an orthonormal family of $L^2(0,4\pi t)$ so by taking the scalar product of $L^2(0,4\pi t)$ with $e^{it\Delta}\delta_k$ we obtain
$$i\partial_t A_k(t)=\int_0^{4\pi t} \mathcal N(\sum_{j\in\mathbb Z}A_j(t)\frac{e^{i\frac{(x-j)^2}{4t}}}{\sqrt{4\pi it}})\frac{-e^{i\frac{(x-k)^2}{4t}}}{\overline{\sqrt{4\pi it}}}\,dx-\frac{M}{4\pi t}A_k(t).$$
Note that as $s>\frac 12$ we have $\{A_j\}\in l^{2,s}\subset l^1$ and we can develop the cubic power to get 
\begin{equation}\label{2corners0}i\partial_t A_k(t)=\frac{1}{8\pi t}\sum_{k-j_1+j_2-j_3=0}e^{-i\frac{k^2-j_1^2+j_2^2-j_3^2}{4t}}A_{j_1}(t)\overline{A_{j_2}(t)}A_{j_3}(t)- \frac{M}{4\pi t}A_k(t).\end{equation}

We note  already that for a sequence of real numbers $a(k)$ we have:
\begin{equation}\label{cons}\partial_t \sum_k a(k)|A_k(t)|^2=\frac{1}{4\pi t}\Im \, \sum_{k-j_1+j_2-j_3=0}a(k)e^{-i\frac{k^2-j_1^2+j_2^2-j_3^2}{4t}}A_{j_1}(t)\overline{A_{j_2}(t)}A_{j_3}(t)\overline{A_k(t)}\end{equation}
$$=\frac{1}{8\pi ti}\left( \sum_{k-j_1+j_2-j_3=0}a(k)e^{-i\frac{k^2-j_1^2+j_2^2-j_3^2}{4t}}A_{j_1}(t)\overline{A_{j_2}(t)}A_{j_3}(t)\overline{A_k(t)}\right.$$
$$\left.-\sum_{j_3-j_2+j_1-k=0}a(k)e^{-i\frac{j_3^2-j_2^2+j_1^2-k^2}{4t}}A_{j_2}(t)\overline{A_{j_1}(t)}A_{k}(t)\overline{A_{j_3}(t)}\right)$$
$$=\frac{1}{8\pi ti}\sum_{k-j_1+j_2-j_3=0}(a(k)-a(j_3))e^{-i\frac{k^2-j_1^2+j_2^2-j_3^2}{4t}}A_{j_1}(t)\overline{A_{j_2}(t)}A_{j_3}(t)\overline{A_k(t)}$$
 $$=\frac{1}{16\pi ti}\sum_{k-j_1+j_2-j_3=0}(a(k)-a(j_1)+a(j_2)-a(j_3))e^{-i\frac{k^2-j_1^2+j_2^2-j_3^2}{4t}}A_{j_1}(t)\overline{A_{j_2}(t)}A_{j_3}(t)\overline{A_k(t)}.$$
Therefore the system conserves the ``mass" :
\begin{equation}\label{mass}\sum_k|A_k(t)|^2= \sum_k|A_k(0)|^2,\end{equation}
and the momentum
\begin{equation}\label{moment}\sum_kk|A_k(t)|^2= \sum_kk|A_k(0)|^2.\end{equation}

We split the summation indices of \eqref{2corners0} into the following two sets:
$$NR_k=\{(j_1,j_2,j_3)\in\mathbb Z^3, k-j_1+j_2-j_3=0, k^2-j_1^2+j_2^2-j_3^2\neq0\},$$
$$Res_k=\{(j_1,j_2,j_3)\in\mathbb Z^3, k-j_1+j_2-j_3=0, k^2-j_1^2+j_2^2-j_3^2=0\}.$$
As we are in one dimension, the second set is simply
$$Res_k=\{(k,j,j), (j,j,k), j\in\mathbb Z\},$$
as for $k-j_1+j_2-j_3=0$ we have
$$k^2-j_1^2+j_2^2-j_3^3=2(k-j_1)(j_1-j_2).$$
In particular we get
$$\sum_{k-j_1+j_2-j_3=0}e^{-i\frac{k^2-j_1^2+j_2^2-j_3^2}{4t}}A_{j_1}(t)\overline{A_{j_2}(t)}A_{j_3}(t)=\sum_{j_1,j_2\in\mathbb Z}e^{-i\frac{2(k-j_1)(j_1-j_2)}{4t}}A_{j_1}(t)\overline{A_{j_2}(t)}A_{k-j_1+j_2}(t)$$
$$=\sum_{j_1\neq k}\sum_{j_2\neq j_1}e^{-i\frac{2(k-j_1)(j_1-j_2)}{4t}}A_{j_1}(t)\overline{A_{j_2}(t)}A_{k-j_1+j_2}(t)+\sum_{j_1\neq k}A_{j_1}(t)\overline{A_{j_1}(t)}A_{k}(t)+\sum_{j_2\in\mathbb Z}A_k(t)\overline{A_{j_2}(t)}A_{j_2}(t).$$
Therefore the system \eqref{syst} writes
\begin{equation}\label{2corners1initial}i\partial_t A_k(t)=\frac{1}{8\pi t}\sum_{(j_1,j_2,j_3)\in NR_k}e^{-i\frac{k^2-j_1^2+j_2^2-j_3^3}{4t}}A_{j_1}(t)\overline{A_{j_2}(t)}A_{j_3}(t)\end{equation}
$$+\frac1{8\pi t}A_k(t)(2\sum_j|A_j(t)|^2-|A_k(t)|^2-2M).$$

As we have already noticed, this system conserves the ``mass" $\sum_j|A_j(t)|^2$, so since $M=\sum_j|\alpha_j|^2$, finding a solution for $t>0$ satisfying
\begin{equation}\label{idcond}\underset{t\rightarrow 0}{\lim}|A_j(t)|=|\alpha_j|,\end{equation}  
is equivalent to finding a solution for $t>0$ satisfying also \eqref{idcond}, for the following also ``mass"-conserving system:
\begin{equation}\label{2corners1}i\partial_t A_k(t)=\frac{1}{8\pi t}\sum_{(j_1,j_2,j_3)\in NR_k}e^{-i\frac{k^2-j_1^2+j_2^2-j_3^3}{4t}}A_{j_1}(t)\overline{A_{j_2}(t)}A_{j_3}(t)-\frac 1{8\pi t}|A_k(t)|^2A_k(t).\end{equation}

By doing a change of phase $A_k(t)=e^{i\frac{ |\alpha_k|^2}{4\pi}\log \sqrt{t}}\tilde A_k(t)$ we get as a system
\begin{equation}\label{systtilde}i\partial_t \tilde A_k(t)=f_k(t)-\frac 1{8\pi t}(|\tilde A_k(t)|^2-|\alpha_k|^2)\tilde A_k(t),\end{equation}
where
\begin{equation}\label{nonrestpart}f_k(t)=\frac{1}{8\pi t}\sum_{(j_1,j_2,j_3)\in NR_k}e^{-i\frac{k^2-j_1^2+j_2^2-j_3^3}{4t}}e^{-i\frac{|\alpha_k|^2-|\alpha_{j_1}|^2+|\alpha_{j_2}|^2-|\alpha_{j_3}|^2}{4\pi}\log\sqrt{t}}\tilde A_{j_1}(t)\overline{\tilde A_{j_2}(t)}\tilde A_{j_3}(t).\end{equation}

Now we note that a solution of \eqref{systtilde} satisfies
\begin{equation}\label{massev}\partial_t |\tilde A_k(t)|^2=2\Im (f_k(t)\overline{\tilde A_k(t)}),\end{equation}
so obtaining a solution of \eqref{systtilde} for $t>0$ with 
\begin{equation}\label{idcondtilde}\underset{t\rightarrow 0}{\lim}|\tilde A_k(t)|=|\alpha_k|,\end{equation}  
is equivalent to obtaining a solution for $t>0$ also satisfying \eqref{idcondtilde}, for the following system, that also enjoys \eqref{massev}:
\begin{equation}\label{systtildebis}i\partial_t \tilde A_k(t)=f_k(t)-\frac 1{8\pi t}\int_0^t 2\Im (f_k(\tau)\overline{\tilde A_k(\tau)})d\tau\,\tilde A_k(t).\end{equation}

We recall that we expect solutions behaving as $A_k(t)=e^{i\frac{ |\alpha_k|^2}{4\pi}\log \sqrt{t}}(\alpha_k+R_k(t))$, with $\{R_k\}$ in the space:
\begin{equation}\label{decay} X^\gamma:=\{\{f_k\}\in\mathcal C^1((0,T),l^{2,s}),\,\,\|\{t^{-\gamma}f_k(t)\}\|_{L^\infty(0,T) l^{2,s}}+\|\{t\,\partial_tf_k(t)\}\|_{L^\infty(0,T) l^{2,s}}<\infty\},\end{equation}
with $T$ to be specified later. We also denote
$$\|\{f_k\}\|_{X^\gamma}=\|\{t^{-\gamma}f_k(t)\}\|_{L^\infty(0,T) l^{2,s}}+\|\{t\,\partial_tf_k(t)\}\|_{L^\infty(0,T) l^{2,s}}.$$
To prove the theorem we shall show that we have a contraction on a  suitable chosen ball of size $\delta$  of $X^\gamma$ for the operator $\Phi$ sending $\{R_k\}$ into 
$$\Phi(\{R_k\})=\{\Phi_k(\{R_j\})\},$$
with
$$\Phi_k(\{R_j\})(t)=i\int_0^t g_k(\tau)d\tau-i\int_0^t \int_0^\tau \Im (g_k(s)\overline{(\alpha_{k}+R_{k}(s)})ds\,(\alpha_{k}+R_{k}(\tau))\frac{d\tau}{4\pi\tau},$$ \\
where
$$g_k(t)=\,\frac{1}{8\pi t}\sum_{(j_1,j_2,j_3)\in NR_k}e^{-i\frac{k^2-j_1^2+j_2^2-j_3^2}{4t}}e^{-i\omega_{k,j_1,j_2,j_3}\log \sqrt{t}}(\alpha_{j_1}+R_{j_1}(t))\overline{(\alpha_{j_2}+R_{j_2}(t))}(\alpha_{j_3}+R_{j_3}(t)),$$
and 
$\omega_{k,j_1,j_2,j_3}=\frac{|\alpha_k|^2-|\alpha_{j_1}|^2+|\alpha_{j_2}|^2-|\alpha_{j_3}|^2}{4\pi}$.
\smallskip

Finally we note that in the case of $N$ Dirac masses with coefficients $|\alpha_k|=a$ and equation \eqref{cubicmod}  with $M=(N-\frac 12)a^2$, we get instead of \eqref{2corners1} the equation
\begin{equation}\label{2corners2}i\partial_t A_k(t)=\frac{1}{8\pi t}\sum_{(j_1,j_2,j_3)\in NR_k}e^{-i\frac{k^2-j_1^2+j_2^2-j_3^3}{4t}}A_{j_1}(t)\overline{A_{j_2}(t)}A_{j_3}(t)-\frac 1{8\pi t}(|A_k(t)|^2-|\alpha_k|^2)A_k(t).\end{equation}
Hence we can write $A_k(t)=\alpha_k+R_k(t)$ and the same fixed point argument works for $\{R_k\}$.

\subsection{The fixed point argument estimates} 

\begin{lemma} For $\{R_k\}\in X^\gamma$ with $\|\{R_k\}\|_{X^\gamma}\leq \delta$ we have the following estimates:
\begin{equation}\label{gk}\|\{g_k(t)\}\|_{l^{2,s}}\leq \frac Ct(\|\{\alpha_k\}\|^3_{l^{2,s}}+t^{3\gamma}\delta^3), 
 \end{equation}
\begin{equation}\label{gkint}\|\{\int_0^t g_k(\tau)d\tau\}\|_{l^{2,s}}\leq Ct(\|\{\alpha_{j}\}\|_{l^{2,s}}^3+\|\{\alpha_{j}\}\|_{l^{2,s}}^5+t^{3\gamma}(1+\|\{\alpha_k\}\|_{l^{2,s}}^2)\delta^3+\|\{\alpha_{j}\}\|_{l^{2,s}}^2\delta+t^{2\gamma}\delta^3), 
 \end{equation}
\begin{equation}\label{gkintbis}\|\{\int_0^t g_k(\tau)\overline{(\alpha_{k}+R_{k}(\tau)})d\tau\}\|_{l^{2,s}}\leq Ct(\|\{\alpha_k\}\|_{l^{2,s}}+t^{\gamma}\delta) 
 \end{equation}
 $$\times (\|\{\alpha_{j}\}\|_{l^{2,s}}^3+\|\{\alpha_{j}\}\|_{l^{2,s}}^5+t^{3\gamma}(1+\|\{\alpha_k\}\|_{l^{2,s}}^2)\delta^3+\|\{\alpha_{j}\}\|_{l^{2,s}}^2\delta+t^{2\gamma}\delta^3).$$

\end{lemma}

\begin{proof}

We note first that
$$\{M_j\}\star \{{N_j}\}\star\{P_j\}(k)=\sum_{(j_1,j_2,j_3)\in NR_k\cup Res_k} M_{j_1}N_{j_2}P_{j_3},$$
so in particular 
\begin{equation}\label{conv}\left| \sum_{(j_1,j_2,j_3)\in NR_k} M_{j_1}N_{j_2}P_{j_3}\right|\leq \{|M_j|\}\star \{{|N_j|}\}\star\{|P_j|\}(k).\end{equation}
We shall frequently use the following inequality:
\begin{equation}\label{young}\|\{M_j\}\star \{{N_j}\}\star\{P_j\}\|_{l^{\infty}}+\|\{M_j\}\star \{{N_j}\}\star\{P_j\}\|_{l^{2,s}}\leq C\|\{M_j\}\|_{l^{2,s}}\|\{N_j\}\|_{l^{2,s}}\|\{P_j\}\|_{l^{2,s}}.\end{equation}
The first part follows from $l^{2,s}\subset l^1$ and the second part follows using also the weighted Young argument on two series:
$$\|\{M_j\}\star \{{N_j}\}\|_{l^{2,s}}\leq C\|\{M_j\}\star \{(1+|j|)^{s}N_j\}\|_{l^{2}}+C\|\{(1+|j|)^{s}M_j\}\star\{N_j\}\|_{l^{2}}$$
$$\leq C\|\{M_j\}\|_{l^{1}}\|\{(1+|j|)^{s}N_j\}\|_{l^{2}}+C\|\{(1+|j|)^{s}M_j\}\|_{l^{2}}\|\{N_j\}\|_{l^{1}}\leq C\|\{M_j\}\|_{l^{2,s}}\|\{N_j\}\|_{l^{2,s}}.$$


Therefore by \eqref{conv} we have 
$$|g_k(t)|\leq\frac{C}{t}\sum_{(j_1,j_2,j_3)\in NR_k}(|\alpha_{j_1}|+|R_{j_1}(t)|)(|\alpha_{j_2}|+|R_{j_2}(t)|)(|\alpha_{j_3}|+|R_{j_3}(t)|)$$
$$\leq \frac{C}{t}\{|\alpha_{j}|+|R_{j}(t)|\}\star \{|\alpha_{j}|+|R_{j}(t)|\}\star\{|\alpha_{j}|+|R_{j}(t)|\}(k),$$
and by \eqref{young} we get \eqref{gk}.

To estimate $\int_0^t g_k(\tau)d\tau$ we perform an integration by parts to get advantage of the non-resonant phase and to obtain integrability in time:
\begin{equation}\label{05}i\int_0^t g_k(\tau)d\tau=t\sum_{(j_1,j_2,j_3)\in NR_k}\frac{e^{-i\frac{k^2-j_1^2+j_2^2-j_3^2}{4t}}e^{-i\omega_{k,j_1,j_2,j_3}\log \sqrt{t}}}{\pi (k^2-j_1^2+j_2^2-j_3^2)}(\alpha_{j_1}+R_{j_1}(t))\overline{(\alpha_{j_2}+R_{j_2}(t))}(\alpha_{j_3}+R_{j_3}(t))
\end{equation}
$$-\int_0^t \sum_{(j_1,j_2,j_3)\in NR_k}\frac{e^{-i\frac{k^2-j_1^2+j_2^2-j_3^2}{4\tau}}}{\pi (k^2-j_1^2+j_2^2-j_3^2)}$$
$$\times \partial_\tau(\tau e^{-i\omega_{k,j_1,j_2,j_3}\log \sqrt{\tau}}(\alpha_{j_1}+R_{j_1}(\tau))\overline{(\alpha_{j_2}+R_{j_2}(\tau))}(\alpha_{j_3}+R_{j_3}(\tau)))\,d\tau.$$
Indeed, for fixed $t$, this computation is justified by considering, for $0<\eta<t$, the quantity $I_k^{\eta}(t)$  defined as $\Phi_k^1(\{R_j\})(t)$ but with the integral in time from $\eta$  instead of $0$. More precisely, $I_k^{\eta}(t)$ is well defined as the integrand can be upper-bounded using \eqref{conv} and \eqref{young} by the function $C\frac {\|\{\alpha_j\}\|^3_{l^{2,s}}+\|\{R_j(\tau)\}\|^3_{l^{2,s}}} \tau$ which is integrable on $(\eta,t)$. In particular the discrete summation commutes with the integration in time. Performing then integrations by parts on $I_k^{\eta}(t)$ as above, we obtain for  $I_k^{\eta}(t)$ an expression that yields as $\eta\rightarrow 0$ the above expression for $\int_0^t g_k(\tau)d\tau$. \\
We obtain, in view of \eqref{young}, and on the fact that on the resonant set $|k^2-j_1^2+j_2^2-j_3^2|\geq 1$,
$$|\int_0^t g_k(\tau)d\tau|\leq Ct\{|\alpha_{j}|+|R_{j}(t)|\}\star \{|\alpha_{j}|+|R_{j}(t)|\}\star\{|\alpha_{j}|+|R_{j}(t)|\}(k)$$
$$+C(1+\|\{\alpha_k\}\|_{l^{\infty}}^2)\int_0^t \{|\alpha_{j}|+|R_{j}(\tau)|\}\star \{|\alpha_{j}|+|R_{j}(\tau)|\}\star\{|\alpha_{j}|+|R_{j}(\tau)|\}(k)\,d\tau$$
$$+C\int_0^t  \{|\tau \,\partial_\tau R_{j}(\tau)|\}\star \{|\alpha_{j}|+|R_{j}(\tau)|\}\star\{|\alpha_{j}|+|R_{j}(\tau)|\}(k)\,d\tau.$$
We perform Cauchy-Schwarz in the integral terms to get the squares for the discrete variable and we sum using \eqref{young}:
$$\|\{\int_0^t g_k(\tau)d\tau\}\|_{l^{2,s}}^2 \leq Ct^2\,  \|\{\alpha_{j}+R_{j}(t)\}\|_{l^{2,s}}^6+C(1+\|\{\alpha_k\}\|_{l^{2,s}}^4)\,t \int_0^t \|\{\alpha_{j}+R_{j}(\tau)\}\|_{l^{2,s}}^6d\tau$$
$$+Ct \int_0^t \|\{\alpha_{j}+R_{j}(\tau)\}\|_{l^{2,s}}^4\|\{\tau\partial_\tau R_j(\tau)\}\|_{l^{2,s}}^2d\tau.$$
Therefore we get \eqref{gkint}:
$$\|\{\int_0^t g_k(\tau)d\tau\}\|_{ l^{2,s}} \leq Ct(  \|\{\alpha_{j}\}\|_{l^{2,s}}^3+\|\{R_{j}(t)\}\|_{l^{2,s}}^3+  \|\{\alpha_{j}\}\|_{l^{2,s}}^5)$$
$$+C(1+\|\{\alpha_k\}\|_{l^{2,s}}^2)t^{1+3\gamma}\|\{\tau^{-\gamma}R_{j}(\tau)\}\|_{L^\infty(0,T),l^{2,s}}^3$$
$$+Ct \|\{\alpha_{j}\}\|_{l^{2,s}}^2\|\{\tau\partial_\tau R_{j}(\tau)\}\|_{L^\infty(0,T),l^{2,s}}+Ct^{1+2\gamma}\|\{\tau^{-\gamma}R_{j}(\tau)\}\|_{L^\infty(0,T),l^{2,s}}^2\|\{\tau\partial_\tau R_{j}(\tau)\}\|_{L^\infty(0,T),l^{2,s}}$$
$$\leq Ct(\|\{\alpha_{j}\}\|_{l^{2,s}}^3+\|\{\alpha_{j}\}\|_{l^{2,s}}^5)+Ct^{1+3\gamma}(1+\|\{\alpha_k\}\|_{l^{2,s}}^2)\delta^3+Ct\|\{\alpha_{j}\}\|_{l^{2,s}}^2\delta+Ct^{1+2\gamma}\delta^3.$$

The last estimate \eqref{gkintbis} is obtained the same way as \eqref{gkint}, by adding in the computations the extra-term $\alpha_k+R_k(\tau)$ and by upper-bounding it in modulus by $\|\{\alpha_k\}\|_{l^{2,s}}+\tau^{\gamma}\delta$.
 \end{proof}
 
 We now use \eqref{gk} and \eqref{gkintbis} to get
 $$\|\{\partial_t\Phi_k(\{R_j\})(t)\}\|_{l^{2,s}}\leq \|\{ g_k(t)d\tau\}\|_{l^{2,s}}+\|\{ \int_0^t \Im (g_k(s)\overline{(\alpha_{k}+R_{k}(s)})ds\,(\alpha_{k}+R_{k}(t))\}\|_{l^{2,s}}\frac{C}{t}$$
$$\leq \frac Ct(\|\{\alpha_k\}\|^3_{l^{2,s}}+t^{3\gamma}\delta^3)+C(\|\{\alpha_k\}\|_{l^{2,s}}+t^{\gamma}\delta)^2$$
 $$\times (\|\{\alpha_{j}\}\|_{l^{2,s}}^3+\|\{\alpha_{j}\}\|_{l^{2,s}}^5+t^{3\gamma}(1+\|\{\alpha_k\}\|_{l^{2,s}}^2)\delta^3+\|\{\alpha_{j}\}\|_{l^{2,s}}^2\delta+t^{2\gamma}\delta^3).$$
On the other hand, 
 $$|\{\Phi_k(\{R_j\})(t)\}|\leq \left|\int_0^t g_k(\tau)d\tau\right|+\left|\int_0^t\int_0^\tau \Im (g_k(s)\overline{(\alpha_{k}+R_{k}(s)})ds\,(\alpha_{k}+R_{k}(\tau))\frac{d\tau}{4\pi\tau}\right|,$$
 so by Cauchy-Schwarz
 $$|\{\Phi_k(\{R_j\})(t)\}|^2\leq C\left|\int_0^t g_k(\tau)d\tau\right|^2+C\sqrt{t}\int_0^t\left|\int_0^\tau \Im (g_k(s)\overline{(\alpha_{k}+R_{k}(s)})ds\right|^2\,(|\alpha_{k}|^2+|R_{k}(\tau)|^2)\frac{d\tau}{\tau^\frac 32}.$$
Now we use \eqref{gkint} and \eqref{gkintbis} to get
$$\|\{\Phi_k(\{R_j\})(t)\}\|_{l^{2,s}}\leq Ct(\|\{\alpha_{j}\}\|_{l^{2,s}}^3+\|\{\alpha_{j}\}\|_{l^{2,s}}^5+t^{3\gamma}(1+\|\{\alpha_k\}\|_{l^{2,s}}^2)\delta^3+\|\{\alpha_{j}\}\|_{l^{2,s}}^2\delta+t^{2\gamma}\delta^3)$$
$$+Ct(\|\{\alpha_k\}\|_{l^{2,s}}+t^{\gamma}\delta)^2
$$ $$\times (\|\{\alpha_{j}\}\|_{l^{2,s}}^3+\|\{\alpha_{j}\}\|_{l^{2,s}}^5+t^{3\gamma}(1+\|\{\alpha_k\}\|_{l^{2,s}}^2)\delta^3+\|\{\alpha_{j}\}\|_{l^{2,s}}^2\delta+t^{2\gamma}\delta^3).$$

 Summarizing, we have obtained
 \begin{equation}\label{stabnonres}\|\{\Phi(\{R_k\})\}\|_{X^\gamma}\leq C(\|\{\alpha_k\}\|^3_{l^{2,s}}+T^{3\gamma}\delta^3)+CT(\|\{\alpha_k\}\|_{l^{2,s}}+T^{\gamma}\delta)^2\end{equation}
 $$\times (\|\{\alpha_{j}\}\|_{l^{2,s}}^3+\|\{\alpha_{j}\}\|_{l^{2,s}}^5+T^{3\gamma}(1+\|\{\alpha_k\}\|_{l^{2,s}}^2)\delta^3+\|\{\alpha_{j}\}\|_{l^{2,s}}^2\delta+T^{2\gamma}\delta^3)$$ 
$$+CT^{1-\gamma}(\|\{\alpha_{j}\}\|_{l^{2,s}}^3+\|\{\alpha_{j}\}\|_{l^{2,s}}^5+T^{3\gamma}(1+\|\{\alpha_k\}\|_{l^{2,s}}^2)\delta^3+\|\{\alpha_{j}\}\|_{l^{2,s}}^2\delta+T^{2\gamma}\delta^3)$$
$$+CT^{1-\gamma}(\|\{\alpha_k\}\|_{l^{2,s}}+T^{\gamma}\delta)^2$$ 
$$\times (\|\{\alpha_{j}\}\|_{l^{2,s}}^3+\|\{\alpha_{j}\}\|_{l^{2,s}}^5+T^{3\gamma}(1+\|\{\alpha_k\}\|_{l^{2,s}}^2)\delta^3+\|\{\alpha_{j}\}\|_{l^{2,s}}^2\delta+T^{2\gamma}\delta^3).$$
 In view of \eqref{stabnonres}, we can choose $\delta$ in terms of $\|\{\alpha_{j}\}\|_{l^{2,s}}$, and  $T$ small with respect to $\|\{\alpha_{j}\}\|_{l^{2,s}}$ and $\gamma$, to obtain the stability estimate
$$\|\{\Phi(\{R_k\})\}\|_{X^\gamma}<\delta.$$
The contraction estimate is obtained in the same way as the stability one. 
As a conclusion the fixed point argument is closed and this settles the local in time existence of the solutions of Theorem \ref{thdiracs}. 

\begin{rem}\label{remlow}
We notice that in \eqref{05} we just upper-bounded the inverse of the non-resonant phase by $1$. One can actually exploit this decay in the discrete summations to relax the assumptions on the initial data. More precisely, for $1\leq p<\infty$ one can use:
$$\left\|\sum_{(j_1,j_2,j_3)\in NR_k}\frac{M_{j_1}N_{j_2}P_{j_3}}{k^2-j_1^2+j_2^2-j_3^2}\right\|_{l^p}^p=\sum_k\left(\sum_{j_1,j_2; j_1\notin\{k,j_2\}}\frac{M_{j_1}N_{j_2}P_{k-j_1+j_2}}{|j_1-j_2||k-j_1|}\right)^p$$
$$\leq C\sum_k\left(\sum_{j_1,j_2}|M_{j_1}|^p|M_{j_2}|^p|M_{k-j_1+j_2}|^p\right)\left(\sum_{j_1,j_2}\frac 1{(1+|j_1-j_2|)^q(1+|k-j_1|)^q}\right)^\frac pq,$$
where $q$ is the conjugate exponent of $p$. As $1\leq p<\infty$ we have $q>1$ so
$$\left\|\sum_{(j_1,j_2,j_3)\in NR_k}\frac{M_{j_1}N_{j_2}P_{j_3}}{k^2-j_1^2+j_2^2-j_3^2}\right\|_{l^p}\leq \|\{M_j\}\|_{l^p}\|\{N_j\}\|_{l^p}\|\{P_j\}\|_{l^p}.$$

\end{rem}

 \subsection{Global in time extension} 
We consider the local in time solution constructed previously. In the case $s= 1$ we shall prove that the growth of $\|\{\alpha_j+R_j(t)\}\|_{L^\infty(0,T) l^{2,1}}$ is controlled, so we can extend the solution globally in time. Global existence for $s>1$ is obtained by considering the $l^{2,1}$ global solution and proving the persistency of the regularity $l^{2,s}$.\\

We shall use \eqref{cons} with $a(k)=k^2$ to get a control of the ``energy":
 $$\partial_t \sum_k k^2|A_k(t)|^2=\mp\frac{1}{16\pi t}\sum_{k-j_1+j_2-j_3=0}(k^2-j_1^2+j_2^2-j_3^2)e^{-i\frac{k^2-j_1^2+j_2^2-j_3^2}{4t}}A_{j_1}(t)\overline{A_{j_2}(t)}A_{j_3}(t)\overline{A_k(t)}$$
 $$=\pm\frac{it}{4\pi }\sum_{k-j_1+j_2-j_3=0}\partial_t\left(e^{-i\frac{k^2-j_1^2+j_2^2-j_3^2}{4t}}\right)A_{j_1}(t)\overline{A_{j_2}(t)}A_{j_3}(t)\overline{A_k(t)}.$$
By integrating from $0$ to $t$ and then using integrations by parts we get
 $$\sum_k k^2|A_k(t)|^2\leq \sum_k k^2|A_k(0)|^2+Ct\sum_{k-j_1+j_2-j_3=0}|A_{j_1}(t)A_{j_2}(t)A_{j_3}(t)A_k(t)|$$
 $$+ C\int_0^t\sum_{k-j_1+j_2-j_3=0}|\partial_\tau(\tau A_{j_1}(\tau)\overline{A_{j_2}(\tau)}A_{j_3}(\tau)\overline{A_k(\tau))}|d\tau$$
 $$\leq \|\{\alpha_j\}\|_{l^{2,1}}^2+Ct\sum_k (|A_j(t)|\star |A_j(t)|\star |A_j(t)|)(k)|A_k(t)|+ \int_0^t\sum_k (|A_j(\tau)|\star |A_j(\tau)|\star |A_j(\tau)|)(k)|A_k(\tau)|d\tau$$
 $$+ \int_0^t\sum_k (\partial_\tau |A_j(\tau)|\star |A_j(\tau)|\star |A_j(\tau)|)(k)|A_k(\tau)|\tau d\tau+ \int_0^t\sum_k (|A_j(\tau)|\star |A_j(\tau)|\star |A_j(\tau)|)(k)\partial_\tau |A_k(\tau)|\tau d\tau.$$
 We shall use now the following estimate, based on Cauchy-Schwartz inequality, Young and H\"older estimates for weak $l^p$ spaces, and the fact that $\{j^{-\frac 12}\}\in l^2_w$:
 $$|\sum_k \{M_j\}\star \{{N_j}\}\star\{P_j\}(k) R_k|\leq \|\{M_j\}\star \{{N_j}\}\star\{P_j\}\|_{l^2}\|R_j\|_{l^2}\leq C\|\{M_j\}\|_{l^1_w}\|\{{N_j}\}\|_{l^1_w}\|\{P_j\}\|_{l^2}\|R_j\|_{l^2}$$
 $$\leq C\|\{M_j\, j^\frac 12\}\|_{l^2_w}\|\{{N_j\,j^\frac 12}\}\|_{l^2_w}\|\{P_j\}\|_{l^2}\|R_j\|_{l^2}$$
 $$\leq C\|\{M_j\}\|_{l^2}^\frac 12\|\{M_j\}\|_{l^{2,1}}^\frac 12\|\{N_j\}\|_{l^2}^\frac 12\|\{N_j\}\|_{l^{2,1}}^\frac 12\{P_j\}\|_{l^2}\|R_j\|_{l^2}$$
 to obtain
 $$\|\{A_j(t)\}\|_{l^{2,1}}^2\leq  \|\{\alpha_j\}\|_{l^{2,1}}^2+Ct\|\{A_j(t)\}\|_{l^2}^3\|\{A_j(t)\}\|_{l^{2,1}}$$
 $$+\int_0^t \|\{A_j(\tau)\}\|_{l^2}^3\|\{A_j(\tau)\}\|_{l^{2,1}}d\tau$$
 $$+ \int_0^t \|\{\partial_\tau A_j(\tau)\}\|_{l^2}\|\{A_j(\tau)\}\|_{l^2}^2\|\{A_j(\tau)\}\|_{l^{2,1}}\tau d\tau.$$
 Now we notice that for system \eqref{2corners0} we get
 $$\|\{\partial_\tau A_j(t)\}\|_{l^2}\leq \frac{C}{t} (\|\{A_j(t)\}\star \{A_j(t)\}\star\{A_j(t)\}\|_{l^2}+\|\{A_j(t)\}\|_{l^2})$$
 $$\leq \frac{C}{t} (\|\{A_j(t)\}\|_{l^{2,1}}\|\{A_j(t)\}\|_{l^2}^2+\|\{A_j(t)\}\|_{l^2}).$$ 
 By using also the conservation of ``mass" \eqref{mass} we finally obtain
 $$\|\{A_j(t)\}\|_{l^{2,1}}^2\leq  \|\{\alpha_j\}\|_{l^{2,1}}^2+Ct\|\{\alpha_j\}\|_{l^2}^3\|\{A_j(t)\}\|_{l^{2,1}}$$
 $$+\int_0^t \|\{\alpha_j\}\|_{l^2}^3\|\{A_j(\tau)\}\|_{l^{2,1}}d\tau+ \int_0^t \|\{\alpha_j\}\|_{l^2}^4\|\{A_j(\tau)\}\|_{l^{2,1}}^2d\tau.$$
We thus obtain by Gr\"onwall's inequality a control of the growth of $\|A_j(t)\|_{l^{2,1}}$, so the local solution can be extended globally and the proof of Theorem \eqref{thdiracs} is finished.

\subsection{Cases of Dirac masses not necessary located at integer numbers}

Some cases of Dirac masses, not necessary located at integer numbers, were treated in \cite{K} and can be extended here to the cubic case. We denote for doubly indexed sequences
$$ \|\{\alpha_{k,\tilde k}\}\|^2_{l^{2,s}}:=\sum_{k,\tilde k\in\mathbb Z}(1+|k|+|\tilde k|)^{2s}|\alpha_{k,\tilde k}|^2.$$
We note that a distribution $f=\sum_{k\in\mathbb Z}\alpha_{k,\tilde k}\delta_{ak+b\tilde k}$ satisfies
$$\hat f(\xi)=\widehat{\sum_{k\in\mathbb Z}\alpha_{k,\tilde k}\delta_{ak+b\tilde k}}(\xi)=\sum_{k\in\mathbb Z}\alpha_{k,\tilde k}e^{-i\xi(ak+b\tilde k)},$$
that can be seen as the restriction to $\xi_1=\xi_2=\xi$ of 
$$\sum_{k\in\mathbb Z}\alpha_{k,\tilde k}e^{-i\xi_1ak-i\xi_2b\tilde k},$$
which is the Fourier transform of
$$E(f):=\sum_{k\in\mathbb Z}\alpha_{k,\tilde k}\delta_{(ak,b\tilde k)}.$$
We denote
$$H^s_{pF;a,b}:=\{u\in\mathcal S'(\mathbb R^2),\,\, \hat u(\xi_1+\frac{2\pi}{a},\xi_2)= \hat u(\xi_1,\xi_2+\frac{2\pi}{b})=\hat u(\xi_1,\xi_2), \hat u\in H^s((0,\frac{2\pi}{a})\times (0,\frac{2\pi}{b}))\},$$
and $$\|f\|_{H^{s,diag}_{pF;a,b}}=\|\widehat {E(f)}\|_{H^s((0,\frac{2\pi}{a})\times (0,\frac{2\pi}{b}))}.$$

\begin{theorem}\label{thdiracs3} 
Let $s>\frac 12$, $T>0$ and $\frac 12<\gamma<1$.  Let $a,b\in\mathbb R^*$ such that $\frac ab\notin\mathbb Q$. We consider the 1-D cubic NLS equation:
\begin{equation}\label{cubicmod3}i\partial_t u +\Delta u\pm\frac 12(|u|^2-\frac{M}{2\pi t})u=0.
\end{equation}
with $M=\sum_{k,\tilde k\in\mathbb Z}|\alpha_{k,\tilde k}|^2$ and $\#\{(k,\tilde k),\alpha_{k,\tilde k}\neq 0\}<\infty $. There exists $\epsilon_0>0$ such that if $\|\{\alpha_{k,\tilde k}\}\|_{l^{2,s}}\leq \epsilon_0$ then we have $T>0$ and a unique solution on $(0,T)$ of the form  
\begin{equation}\label{ansatzcubic3}u(t)=\sum_{k,\tilde k\in\mathbb Z}e^{\mp i\frac{|\alpha_{k,\tilde k}|^2}{4\pi}\log \sqrt{t}}(\alpha_{k,\tilde k}+R_{k,\tilde k}(t))e^{it\Delta}\delta_{ak+b\tilde k},\end{equation}
with the decay 
\begin{equation}\label{decayansatzcubic3}\sup_{0<t<T}t^{-\gamma}\|\{R_{k,\tilde k}(t)\}\|_{l^{2,s}}+t\,\|\{\partial_t R_{k,\tilde k}(t)\}\|_{l^{2,s}}<C.\end{equation}
Moreover, considering an initial data a finite sum of N Dirac masses
$$u(0)=\sum_{k\in\mathbb Z}\alpha_{k,\tilde k}\delta_{ak+b\tilde k},$$
with coefficients of same modulus $|\alpha_{k,\tilde k}|=a$ and equation \eqref{cubicmod3} normalized with $M=(N-\frac 12)a^2$, we have a unique solution on $(-T,T)$
$$u(t)=e^{it\Delta}u(0)\pm ie^{it\Delta}\int_0^te^{-i\tau\Delta}\left(\left(|u(\tau)|^2-\frac{M}{2\pi \tau}\right)u(\tau)\right)\,\frac{d\tau}2,$$
such that $\widehat{E(e^{-it\Delta}u(t))}\in\mathcal C^1((-T,T),H^s((0,\frac{2\pi}{a})\times (0,\frac{2\pi}{b})))$ with 
$$\|e^{-it\Delta}u(t)-u(0)\|_{H^{s,diag}_{pF;a,b}}\leq C t^\gamma, \quad\forall t\in(-T,T).$$

\end{theorem}


The new phenomenon here is that if for instance the initial data is the sum of three Dirac masses located at $0,a$ and $b$ then we see small effects on the dense subset on $\mathbb R$ given by the group $a\mathbb Z+ b\mathbb Z$. Another difference with respect to the previous case is that the non-resonant phases can approach zero so we shall perform integration by parts from the phase only on the free term. Due to this small divisor problem we impose on one hand only a finite number of Dirac masses at time $t=0$, and on the other hand a smallness condition on the data.

The proof of Theorem \ref{thdiracs3} goes similarly to the one of Theorem \ref{thdiracs}, by plugging the ansatz \eqref{ansatzcubic3} into equation \eqref{cubicmod3} to get by using the orthogonality of the family $\{\frac{e^{i\frac{(x-ak-b\tilde k)^2}{4t}}}{\sqrt{4\pi it}}\}$ the associated system 
$$i\partial_t A_{k,\tilde k}(t)$$
$$=\mp\frac{1}{8\pi t}\sum_{((j_1,\tilde{j_1}),(j_2,\tilde{j_2}),(j_3,\tilde{j_3}))\in NR_{k,\tilde k}}e^{-i\frac{(ka+\tilde kb)^2-(j_1a+\tilde{j_1}b)^2+(j_2a+\tilde{j_2}b)^2-(j_3a+\tilde{j_3}b)^2}{4t}}A_{j_1,\tilde{j_1}}(t)\overline{A_{j_2,\tilde{j_2}}(t)}A_{j_3,\tilde{j_3}}(t)$$
$$\pm\frac1{8\pi t}A_{k,\tilde k}(t)(2\sum_{j,\tilde j}|A_{j,\tilde j}(t)|^2-|A_{k,\tilde k}(t)|^2-2M),$$
where $NR_{k,\tilde k}$ is the set of indices such that the phase does not vanish i.e. such that $k-j_1+j_2-j_3=0, \tilde k-\tilde{j_1}+\tilde{j_2}-\tilde{j_3}=0$,  $k^2-j^2_1+j^2_2-j^2_3\neq0,$ and $\tilde k-\tilde{j_1}+\tilde{j_2}-\tilde{j_3}\neq0$. 
We have to solve the equivalent ``mass"-conserving system:
\begin{equation}\label{2cornersQ1}i\partial_t A_{k,\tilde k}(t)$$
$$=\mp\frac{1}{8\pi t}\sum_{((j_1,\tilde{j_1}),(j_2,\tilde{j_2}),(j_3,\tilde{j_3}))\in NR_{k,\tilde k}}e^{-i\frac{(ka+\tilde kb)^2-(j_1a+\tilde{j_1}b)^2+(j_2a+\tilde{j_2}b)^2-(j_3a+\tilde{j_3}b)^2}{4t}}A_{j_1,\tilde{j_1}}(t)\overline{A_{j_2,\tilde{j_2}}(t)}A_{j_3,\tilde{j_3}}(t)\end{equation}
$$\mp\frac1{8\pi t}|A_{k,\tilde k}(t)|^2A_{k,\tilde k}(t). $$
We look for solutions of the form $A_{k,\tilde k}(t)=e^{\mp i\frac{|\alpha_{k,\tilde k}|^2}{4\pi}\log \sqrt{t}}(\alpha_{k,\tilde k}+R_{k,\tilde k}(t))$, with $\{R_{k,\tilde k}\}\in Y^\gamma$:
\begin{equation}\label{decayQ} Y^\gamma:=\{\{f_{k,\tilde k}\}\in\mathcal C((0,T);l^{2,s})\}.\end{equation}
As for Theorem \ref{thdiracs}, we make a fixed point argument in a ball of $Y^\gamma$ of size depending on $\|\{\alpha_{k,\tilde k}\}\|_{l^{2,s}}$ for the operator $\Phi$ sending $\{R_{k,\tilde k}\}$ into 
$$\Phi(\{R_{k,\tilde k}\})=\{\Phi_{k,\tilde k}(\{R_{j,\tilde j}\})\},$$
with
$$\Phi_{k,\tilde k}(\{R_{j,\tilde j}\}(t))=\mp i\int_0^t f_{k,\tilde k}(\tau)\,d\tau\pm i\int_0^t\int_0^\tau\Im(f_{k,\tilde k}(s)\overline{(\alpha_{k,\tilde k}+R_{k,\tilde k}(s))}ds(\alpha_{k,\tilde k}+R_{k,\tilde k}(\tau))\frac{d\tau}{4\pi t},$$ \\
where
$$ f_{k,\tilde k}(t)=\sum_{((j_1,\tilde{j_1}),(j_2,\tilde{j_2}),(j_3,\tilde{j_3}))\in NR_{k,\tilde k}}\frac{e^{-i\frac{(ka+\tilde kb)^2-(j_1a+\tilde{j_1}b)^2+(j_2a+\tilde{j_2}b)^2-(j_3a+\tilde{j_3}b)^2}{4t}}}{8\pi t}$$
$$\times e^{-i\frac{|\alpha_{k,\tilde k}|^2-|\alpha_{j_1,\tilde j_1}|^2+|\alpha_{j_2,\tilde j_2}|^2-|\alpha_{j_3,\tilde j_3}|^2}{4\pi}\log t}(\alpha_{j_1,\tilde {j_1}}+R_{j_1,\tilde {j_1}}(t))\overline{(\alpha_{j_2,\tilde {j_2}}+R_{j_2,\tilde {j_2}}(t))}(\alpha_{j_3,\tilde {j_3}}+R_{j_3,\tilde {j_3}}(t)).$$
To avoid issues related to having the non-resonant phase approaching zero, we perform integrations by parts only in the free term involving a finite number of terms, as $\#\{(k,\tilde k),\alpha_{k,\tilde k}\neq 0\}<\infty $. All the remaining terms contain powers of $R_{j,\tilde j}(\tau)$ so we get integrability in time by using the Young inequalities \eqref{young} for double indexed sequences. However, due to the presence of terms linear in $R_{j,\tilde j}(\tau)$ we need to impose a smallness condition on the initial data $\|\{\alpha_{j,\tilde j}\}\|_{l^{2,s}}$. Moreover, from the cubic terms treated without integrations by parts as previously, we need to impose $\gamma>\frac 12$. The control of $\|\{t\partial_t R_{k,\tilde k}(t)\}\|_{L^\infty(0,T) l^{2,s}}$ is easily obtained a-posteriori, once a solution is constructed in $Y^\gamma$.



\section{The Talbot effect}\label{sectTalbot}
The Talbot effect for the linear and nonlinear  Schr\"odinger equations on the torus with initial data given by functions with bounded variation has been largely studied (\cite{Be},\cite{Os},\cite{Ro},\cite{Ta},\\ \cite{ET},\cite{ChErTz}). Here we place ourselves in a more singular setting on $\mathbb R$, and get closer to the Talbot effect observed in optics (see for example \cite{BeKl}) which is typically modeled with Dirac combs as we consider in this paper.

As a consequence of Theorem \ref{thdiracs} the solution $u(t)$ of equation \eqref{cubicmod} with initial data 
$$u(0)=\sum_{k\in\mathbb Z}\alpha_k\delta_k,$$
with $|\alpha_k|=a$ behaves for small times like $e^{it\Delta}u_0$.  We compute first the linear evolution $e^{it\Delta}u_0$ which display a Talbot effect.

\begin{prop}\label{thTalbot}{\bf{(Talbot effect for linear evolutions)}} Let $p\in\mathbb N$ and $u_0$ with $\hat{u_0}$ $2\pi-$periodic. For all $t_{p,q}=\frac 1{2\pi}\frac pq$ with $q$ odd and for all $x\in\mathbb R$ we have
\begin{equation}\label{repr1}e^{it_{p,q}\Delta}u_0(x)=\frac {1}{\sqrt{q}}\int_0^{2\pi}\hat{u_0}(\xi)e^{-it_{p,q}\xi^2+ix\xi}\sum_{l\in\mathbb Z}\sum_{m=0}^{q-1}e^{i\theta_{m,p,q}}\delta(x-2t_{p,q}\,\xi-l-\frac mq)\,d\xi,\end{equation}
for some $\theta_{m,p,q}\in\mathbb R$. 
We suppose now that moreover $\hat u_0$ is located modulo $2\pi$ only in a neighborhood of zero of radius less than $\eta \frac {\pi}{ p}$ with $0<\eta<1$. For a given $x\in\mathbb R$ we define
$$\xi_{x} :=\frac {\pi q}p \,dist\left (x,\frac1q\mathbb Z\right) \in[0,\frac \pi p).$$
Then there exists $\theta_{x,p,q}\in\mathbb R$ such that
\begin{equation}\label{repr2}e^{it_{p,q}\Delta}u_0(x)=\frac {1}{\sqrt{q}} \,\hat{u_0}(\xi_x)\, e^{-it_{p,q}\,\xi_x^2+ix\,\xi_x+i\theta_{x,p,q}}.\end{equation}
In particular $|e^{it_{p,q}\Delta}u_0(l+\frac mq)|=|e^{it_{p,q}\Delta}u_0(0)|$ and if $x$ is at distance larger than $\frac \eta{q}$ from $\frac1q\mathbb Z$ then $e^{it_{p,q}\Delta}u_0(x)$ vanishes. 

Moreover, the solution can concentrate near $\frac1q\mathbb Z$ in the sense that there is a family of initial data $u_0^\lambda=\sum _{k\in\mathbb Z}\alpha_k^\lambda\delta_k$ and $C>0$ such that 
\begin{equation}\label{conc}
\left|\frac{e^{it_{p,q}\Delta}u_0^\lambda(0)}{e^{it_{p,q}\Delta}\alpha_0^\lambda\delta_0(0)}\right|\overset{\lambda\rightarrow\infty}{\longrightarrow}\infty.
\end{equation}
\end{prop}

We note here that the data $u_0=\sum _{k\in\mathbb Z}\delta_k$ enter the above setting of the $2\pi-$periodicity in Fourier and localization in Fourier, as $\hat{u_0}=u_0=\sum _{k\in\mathbb Z}\delta_k$ (see last section where Poisson summation formula is recalled). Therefore $e^{it\Delta}u_0(x)=0$ for $x\notin \frac1q\mathbb Z$, and is a Dirac mass otherwise, which is the classical Talbot effect. However, due to the non-summation of its coefficients, this kind of data does not enter our nonlinear theorem framework. Nevertheless, the concentration phenomena \eqref{conc} is obtained by taking a sequence of initial data $\{u_0^\lambda\}$ that concentrates in Fourier variable near integers.

Proposition \ref{thTalbot} insures the persistence of the Talbot effect at the nonlinear level. 

\begin{prop}\label{proptalbotnl}{\bf{(Talbot effect for nonlinear evolutions)}} Let $p\in\mathbb N$, $\epsilon\in (0,1)$ and $q_\epsilon$ such that  $\epsilon^2\sqrt{q_\epsilon}\log q_\epsilon<\frac 12$; in particular  $q_\epsilon\overset{\epsilon\rightarrow 0}{\longrightarrow}+\infty$.

Let $u_0$ be such that $\hat u_0$ is a $2\pi-$periodic, located modulo $2\pi$ only in a neighborhood of zero of radius less than $\eta \frac {\pi}{ p}$ with $0<\eta<1$ and having Fourier coefficients such that $\|\{\alpha_k\}\|_{l^{2,s}}\leq \epsilon$ for some $s>\frac 12$. Let $u(t,x)$ be the solution of \eqref{cubicmod} obtained in Theorem \ref{thdiracs} from $\{\alpha_k\}$. Then for all $t_{p,q}=\frac 1{2\pi}\frac pq$ with $1\leq q\leq q_\epsilon$ odd and for all $x$ at distance larger than $\frac \eta{q}$ from $\frac1q\mathbb Z$ the function $u(t,x)$ almost vanishes in the sense:
\begin{equation}\label{vanishing}|u(t_{p,q},x)|\leq \epsilon.\end{equation}

Moreover, the solution can concentrate near $\frac1q\mathbb Z$ in the sense that there is a family of sequences $\{\alpha_k^\lambda\}$ with $\|\{\alpha_k^\lambda\}\|_{l^{2,s}}\overset{\lambda\rightarrow \infty}{\longrightarrow}0$ such that the corresponding solutions $u_\lambda$ obtained in Theorem \ref{thdiracs} satisfy
\begin{equation}\label{concnonlin}
\left|\frac{u^\lambda(t_{p,q},0)}{e^{it_{p,q}\Delta}\alpha_0^\lambda\delta_0(0)}\right|\overset{\lambda\rightarrow\infty}{\longrightarrow}\infty.
\end{equation}
\end{prop}

\subsection{Proof of Propositions \ref{thTalbot}-\ref{proptalbotnl}}
We start by recalling the Poisson summation formula $\sum _{k\in\mathbb Z}f_k=\sum _{k\in\mathbb Z}\hat f(2\pi k)$ for the Dirac comb:
$$
(\sum _{k\in\mathbb Z}\delta_k)(x)=\sum _{k\in\mathbb Z}\delta(x-k)=\sum_{k\in\mathbb Z}e^{i2\pi kx},$$
as 
$$\widehat{\delta(x-\cdot)}(2\pi k)=\int_{-\infty}^\infty e^{-i2\pi ky}\delta(x-y)\,dy=e^{-i2\pi kx}.$$

The computation of the free evolution with periodic Dirac data is
\begin{equation}\label{perevol}e^{it\Delta}(\sum _{k\in\mathbb Z}\delta_k)(x)=\sum_{k\in\mathbb Z}e^{-it(2\pi k)^2+i2\pi kx}.\end{equation}
For $t=\frac 1{2\pi}\frac pq$ we have (choosing $M=2\pi$ in formulas (37) combined with (42) from \cite{DHV})
\begin{equation}\label{Talbotper}e^{it\Delta}(\sum _{k\in\mathbb Z}\delta_k)(x)=\frac 1{q}\sum_{l\in\mathbb Z}\sum_{m=0}^{q-1}G(-p,m,q)\delta(x-l-\frac mq),\end{equation}
which describes the linear Talbot effect in the periodic setting. Here $G(-p,m,q)$ stands for the Gauss sum
$$G(-p,m,q)=\sum_{l=0}^{q-1}e^{2\pi i\frac{-pl^2+ml}{q}}.$$\\

Now we want to compute the free evolution of data $u_0=\sum _{k\in\mathbb Z}\alpha_k\delta_k$. As $\widehat{u_0}(\xi)=\sum_{k\in\mathbb Z}e^{-ik\xi}$ is $2\pi-$periodic  we have:
$$e^{it\Delta}u_0(x)=\frac 1{2\pi}\int_{-\infty}^\infty e^{ix\xi}e^{-it\xi^2}\hat{u_0}(\xi)\,d\xi=\frac 1{2\pi}\sum_{k\in\mathbb Z}\int_{2\pi k}^{2\pi(k+1)} e^{ix\xi-it\xi^2}\hat{u_0}(\xi)\,d\xi$$
$$=\frac 1{2\pi}\int_0^{2\pi}\hat{u_0}(\xi)\sum_{k\in\mathbb Z}e^{ix(2\pi k+\xi)-it(2\pi k+\xi)^2}\,d\xi=\frac 1{2\pi}\int_0^{2\pi}\hat{u_0}(\xi)e^{-it\xi^2+ix\xi}\sum_{k\in\mathbb Z}e^{-it\,(2\pi k)^2+i2\pi k (x- 2t\xi )}\,d\xi.$$
Therefore,  for $t_{p,q}=\frac 1{2\pi}\frac pq$ we get using \eqref{perevol}-\eqref{Talbotper}:
$$e^{it_{p,q}\Delta}u_0(x)=\frac 1{q}\int_0^{2\pi}\hat{u_0}(\xi)e^{-it_{p,q}\xi^2+ix\xi}\sum_{l\in\mathbb Z}\sum_{m=0}^{q-1}G(-p,m,q)\delta(x-2t_{p,q}\xi-l-\frac mq)\,d\xi.$$
For $q$ odd  $ G(-p,m,q)=\sqrt{q}e^{i\theta_m}$ for some $\theta_{m,p,q}\in\mathbb R$ so we get for $t_{p,q}=\frac 1{2\pi}\frac pq$
$$e^{it_{p,q}\Delta}u_0(x)=\frac {1}{\sqrt{q}}\int_0^{2\pi}\hat{u_0}(\xi)e^{-it_{p,q}\xi^2+ix\xi}\sum_{l\in\mathbb Z}\sum_{m=0}^{q-1}e^{i\theta_{m,p,q}}\delta(x-2t_{p,q}\,\xi-l-\frac mq)\,d\xi.$$
We note that for $0\leq \xi<2\pi$ we have $0\leq 2t\xi<\frac {2p}q$. 
For a given $x\in\mathbb R$ there exists a unique  $l_x\in\mathbb Z$ and a unique $0\leq m_x<q$ such that 
$$x-l_x-\frac {m_x}q \in[0,\frac 1q).$$
We denote
$$\xi_{x} :=\frac {\pi q}p(x-l_x-\frac {m_x}q)=\frac {\pi q}p\,d(x,\frac1q\mathbb Z) \in[0,\frac \pi p).$$
In particular if $\hat u_0$ is located modulo $2\pi$ only in a neighborhood of zero of radius less than $\frac {\pi}{ p}$ then 
$$e^{it_{p,q}\Delta}u_0(x)=\frac {1}{\sqrt{q}} \,\hat{u_0}(\xi_x)\, e^{-it_{p,q}\,\xi_x^2+ix\,\xi_x+i\theta_{m,p,q}},$$
for some $\theta_{x,p,q}\in\mathbb R$. 
If moreover $\hat u_0$ is located modulo $2\pi$ only in a neighborhood of zero of radius less than $\eta \frac {\pi}{ p}$ with $0<\eta<1$, then the above expression vanishes for $x$ at distance larger than $\frac {\eta }{ q}$ from $\frac1q\mathbb Z$.

We are left with proving the concentration effect \eqref{conc} of Proposition \ref{thTalbot}. We shall construct a family of sequences $\{\alpha_k^\lambda\}$ such that $\sum _{k\in\mathbb Z} \alpha_k^\lambda \delta_k$ concentrates in Fourier variable near integers. To this purpose we consider $\psi$ a real bounded function with support in $[-\frac 12,\frac 12]$ and $\psi(0)=1$. We define
$$f^\lambda(\xi)= \lambda^{\beta}\psi(\lambda\xi) ,\forall \xi\in[-\pi,\pi],$$
with $\beta\in\mathbb R$.
Thus we can decompose
$$f^\lambda(\xi)=\sum _{k\in\mathbb Z} \alpha_k^\lambda e^{ik\xi},$$
and consider 
$$u_0^\lambda=\sum _{k\in\mathbb Z} \alpha_k^\lambda \delta_k.$$
In particular, on $[-\pi,\pi]$, we have $\widehat{u_0^\lambda}=f^\lambda$. Given $t_{p,q}=\frac 1{2\pi}\frac pq$, for $\lambda>p$, the restriction of $\widehat{u_0^\lambda}$ to $[-\pi,\pi]$ has support included in a neighborhood of zero of radius less than $\eta\frac \pi p$ for a $\eta\in]0,1[$. We then get by \eqref{repr2}
$$e^{it_{p,q}\Delta}u_0^\lambda(0)=\frac {1}{\sqrt{q}} \,\widehat{u_0^\lambda}(0)\, e^{-it_{p,q}\,\xi_x^2+ix\,\xi_x+i\theta_{m_x}},$$
so 
$$|e^{it_{p,q}\Delta}u_0^\lambda(0)|=\frac {1}{\sqrt{q}} \,|f^\lambda(0)|=\frac {1}{\sqrt{q}} \lambda^{\beta}\psi(0)=\frac {1}{\sqrt{q}} \lambda^{\beta}.$$
On the other hand, at $t_{p,q}=\frac 1{2\pi}\frac pq$ we have
$$|e^{it_{p,q}\Delta}\alpha_0^\lambda\delta_0(0)|=\sqrt{\frac {4q}p} \,|\alpha_0^\lambda|=\sqrt{\frac {4q}p} \,\frac1{2\pi}\left|\int_{-\pi}^\pi f^\lambda(\xi)d\xi\right|=C(\psi)\sqrt{\frac qp} \lambda^{\beta-1}$$
Therefore
$$\left|\frac{e^{it_{p,q}\Delta}u_0^\lambda(0)}{e^{it_{p,q}\Delta}\alpha_0^\lambda\delta_0(0)}\right|=\frac{
\sqrt{p}}{C(\psi)q}\lambda\overset{\lambda\rightarrow\infty}{\longrightarrow}\infty,$$  
and the proof of Proposition \ref{thTalbot} is complete.

Finally, for the first part of Proposition \ref{proptalbotnl}, as the sequence $\{\alpha_k\}$ enters the framework of Theorem \ref{thdiracs},
$$u(t_{p,q},x)=\sum_{k\in\mathbb Z}e^{\mp i\frac{|\alpha_k|^2}{4\pi}\log \sqrt{t_{p,q}}}(\alpha_k+R_k(t_{p,q}))e^{it_{p,q}\Delta}\delta_k(x),$$
so
$$\left|u(t_{p,q},x)-\sum_{k\in\mathbb Z}e^{it_{p,q}\Delta}\alpha_k\delta_k(x)\right|$$
$$\leq \sum_{k\in\mathbb Z}(1-e^{\mp i\frac{|\alpha_k|^2}{4\pi}\log \sqrt{t_{p,q}}})\alpha_k e^{it_{p,q}\Delta}\delta_k(x)+\sum_{k\in\mathbb Z}e^{\mp i\frac{|\alpha_k|^2}{4\pi}\log \sqrt{t_{p,q}}}R_k(t_{p,q})e^{it_{p,q}\Delta}\delta_k(x).$$
From Proposition \ref{thTalbot}, if $x$ is at distance larger than $\frac \eta{q}$ from $\frac1q\mathbb Z$ then $e^{it_{p,q}\Delta}\sum_k\alpha_k\delta_k(x)$ vanishes. Also, from \eqref{stabnonres} we can choose the radius $\delta$ of the fixed point argument for $\{R_k\}$ to be of type $C\|\{\alpha_k\}\|_{l^{2,s}}^3$ so we get
$$|u(t_{p,q},x)|\leq \sum_{k\in\mathbb Z}|1-e^{i\frac{\mp |\alpha_k|^2}{4\pi}\log \sqrt{t_{p,q}}}||\alpha_k| \frac{C}{\sqrt{t_{p,q}}}+C\|\{\alpha_k\}\|_{l^{2,s}}^3t_{p,q}^{\gamma-\frac 12}.$$
If $q$ is such that  $\|\{\alpha_k\}\|_{l^{\infty}}^2\log q<\frac 12$ then we obtain
$$|u(t_{p,q},x)|\leq C\sqrt{q}\log q\sum_{k\in\mathbb Z}|\alpha_k|^3+\frac{C}{q^{\gamma-\frac 12}}\|\{\alpha_k\}\|_{l^{2,s}}^3,$$
and therefore \eqref{vanishing} follows for $C\sqrt{q}\log q\,\epsilon^2<1$.

For the last part of Proposition \ref{proptalbotnl} we proceed as for the last part of Proposition \ref{thTalbot}, and we suppose also that $\psi\in H^s(\mathbb R)$ with $s>\frac 12$, and impose $\beta<\frac 12-s$. Then the condition $\psi\in H^s(\mathbb R)$ insures us that $\{\alpha^\lambda_k\}\in l^{2,s}$, and moreover 
$\|\{\alpha^\lambda_k\}\|_{l^{2,s}}=C(\psi)\lambda^{\beta+s-\frac 12}\overset{\lambda\rightarrow +\infty}{\longrightarrow}0$. Therefore, for $\lambda$ large enough, by using the same estimates as above we obtain for $\frac 12<\gamma<1$
$$|u^\lambda(t_{p,q},0)-e^{it_{p,q}\Delta}u_0^\lambda(0)|\leq C\sqrt{q}\log q\|\{\alpha^\lambda_k\}\|_{l^{2,s}}^3+\frac{C}{q^{\gamma-\frac 12}}\|\{\alpha^\lambda_k\}\|_{l^{2,s}}^3\leq C\sqrt{q}\log q\lambda^{3\beta+3s-\frac 32},$$
so
$$\left|\frac{u^\lambda(t_{p,q},0)}{e^{it_{p,q}\Delta}\alpha_0^\lambda\delta_0(0)}-\frac{e^{it_{p,q}\Delta}u_0^\lambda(0)}{e^{it_{p,q}\Delta}\alpha_0^\lambda\delta_0(0)}\right|\leq C\log q\lambda^{2\beta+3s-\frac 12}.$$
By choosing $\beta=\frac 32(\frac 12-s)^-$ we have $\lambda^{2\beta+3s-\frac 12}\ll \lambda$ so in view of \eqref{conc} the divergence \eqref{concnonlin} follows.


\section{Evolution of polygonal lines through the binormal flow} 
In this section we prove Theorem \ref{brokenline}.
\subsection{Plan of the proof}

We consider equation \eqref{cubicmod} with initial data 
$$u(0)=\sum_{k\in\mathbb Z} \alpha_k\delta_k,$$
where the coefficients $\alpha_k$ will be defined in \S\ref{sectdesign} in a specific way involving geometric quantities characterizing the polygonal line $\chi_0$. 
Then Theorem \ref{thdiracs} gives us a solution $u(t,x)$ on $t>0$. From this smooth solution on $]0,\infty[$ we construct in \S \ref{sectconstr} a smooth solution $\chi(t)$ of the binormal flow on $]0,,\infty[$, that has a limit $\chi(0)$ at $t=0$. 
Now the goal is to prove that the curve $\chi(0)$ is the curve $\chi_0$ modulo a translation and a rotation. This is done in several steps. First we show in \S \ref{sectTlimit} that the tangent vector has a limit at $t=0$. Secondly we show in \S \ref{sectsegm} that $\chi(0)$ is a segment for $x\in]n,n+1[,\forall n\in\mathbb Z$. Then we prove in \S \ref{sectcorners}, by analyzing the frame of the curve through self-similar variables paths, that at points $x=k\in\mathbb Z$ the curve $\chi(0)$ presents a corner of same angle as $\chi_0$. In \S \ref{secttorsion} we recover the torsion of $\chi_0$ by using also a similar analysis for modulated normal vectors in \S \ref{sectN}.   
Therefore we conclude in \S \ref{sectendproof} that $\chi(0)$ and $\chi_0$ coincide modulo a translation and a rotation. By considering the corresponding translated and rotated $\chi(t)$ we obtain the desired binormal flow solution with limit $\chi_0$ at $t=0$. The extension to negative times is done by using time reversibility. 
Uniqueness holds in the class of curves having filament functions of type \eqref{ansatz}. In \S \ref{sectproperties} we describe some properties of the binormal flow solution given by the Theorem \ref{brokenline}.

\subsection{Designing the coefficients of the Dirac masses in geometric terms} \label{sectdesign}
Let $\chi_0(x)$ be a polygonal line paramatrized by archlength, having at least two consecutive corners, located at $x=x_0$ and $x=x_1$. We denote by $\{x_n, n\in L\}\subset\mathbb R$,  the locations of all its corners ordered incresingly: $x_n<x_{n+1}$. Here $L$ stand for a finite or infinite set of consecutive integers including $n=0$. We can characterize such a curve by the location of its corners $\{x_n, n\in L\}\subset\mathbb R$ and by a triple sequence $\{\theta_n,\tau_n,\delta_n\}_{n\in L}$ where $\theta_n\in]0,\pi[$, $\tau_n\in[0,\pi]$ and $\delta_n\in\{-,+\}$, $\tau_{0}=0$, $\delta_{0}=+$, in the following way. 

Let us first denote by $T_n\in\mathbb S^2$ the tangent vector of $\chi_0(x)$ for $x\in]x_n,x_{n+1}[$.  For $n\in L$ we define $\theta_n\in]0,\pi[$ to be the (curvature) angle between $T_{n-1}$ and $T_{n}$. 
We note that given only $T_{n-1}$ and $\theta_n$ we have a $[0,2\pi[$-parameter of possibilities to choose $T_{n}$. We define $\tau_{0}=0$, $\delta_{0}=+$ and for $n\in L$ with $n+1\in L$ and $n\geq 0$ 
we define a (torsion) angle $\tau_{n+1}\in[0,\pi]$ at the corner located at $x_{n+1}$  to be such that
\begin{equation}\label{deftorsionmod}
\cos(\tau_{n+1})=\frac{T_{n-1}\wedge T_n}{|T_{n-1}\wedge T_n|}.\frac{T_n\wedge T_{n+1}}{|T_n\wedge T_{n+1}|}.
\end{equation}
We note now that given only $T_{n-1},T_n$, $\theta_n$ and $\tau_n$ we have two possibilities to choose $T_{n+1}$. Indeed, $T_{n+1}$ is determined by the position of $T_n\wedge T_{n+1}$ in the plane $\Pi_n$ orthogonal to $T_n$, given by the oriented frame $T_{n-1}\wedge T_n$ and $T_n\wedge (T_{n-1}\wedge T_n)$. Therefore we have two possibilities by orienting it with respect to $T_{n-1}\wedge T_n$: by $\tau_n$ or by $-\tau_n$. We define $\delta_{n+1}=+$ if $ (T_{n-1}\wedge T_n)\wedge(T_n\wedge T_{n+1})$ points in the same direction as $T_n$, and 
$\delta_{n+1}=-$ if it points out in the opposite direction. For $n<0$ we define similarly (torsion) angles  $\tau_n\in[0,\pi[$ at the corners located at $x_n$. 

Conversely, given $L$ a set of consecutive integers containing $0$ and $1$, given an increasing sequence $\{x_n, n\in L\}\subset\mathbb R$, and given a triple sequence $\{\theta_n,\tau_n,\delta_n\}_{n\in L}$ where $\theta_n\in]0,\pi[$, $\tau_n\in[0,\pi]$ and $\delta_n\in\{-,+\}$, such that $\tau_{0}=0$, $\delta_{0}=+$
, we can determine a polygonal line $\chi_0$, unique up to rotations and translation, in the following way. We construct first the tangent vectors, then $\chi_0$ is obtained by setting $\chi_0'(x)=T_n$ on $x\in]x_n,x_{n+1}[$. We pick a unit vector and denote it $T_{-1}$. Then we pick a unit vector having an angle $\theta_{0}$ with $T_{-1}$, and we call it $T_{0}$. Then we consider all unit vectors $v$ having an angle $\theta_{1}$ with $T_{0}$. Among them, we choose the two of them such that $T_{0}\wedge v$, that lives in the plane $\Pi_{0}$ orthogonal to $T_{0}$, have an angle $\tau_{1}$ with $T_{-1}\wedge T_{0}$. Eventually, we choose $T_{1}$ to be the one of the two such that if $\delta_{0}=+$ the vector $ (T_{-1}\wedge T_{0})\wedge(T_{0}\wedge v)$ points in the same direction as $T_{0}$, and such that if $\delta_{0}=+$ the vector $ (T_{-1}\wedge T_{0})\wedge(T_{0}\wedge v)$ points in the opposite direction of $T_{0}$. We iterate this process to construct $\chi_0(x)$ on $x>x_0$, and similarly to construct $\chi_0(x)$ on $x<x_0$.

Given $\chi_0$ the polygonal line of the statement, we define $\{x_n, n\in L\}$ the ordered set of its integer corner locations and the corresponding sequence $\{\theta_n,\tau_n,\delta_n\}_{n\in L}$ where $\theta_n\in]0,\pi[$, $\tau_n\in[0,\pi]$ and $\delta_n\in\{-,+\}$, $\tau_{0}=0$, $\delta_{0}=+$. Then we define $\alpha_k=0$ if $k\notin\{x_n, n\in L\} $ and if $k=x_n$ for some $n\in L$ we define $\alpha_k\in\mathbb C$ in the following way. First we set
\begin{equation}\label{moduluscoef}
|\alpha_k|=\sqrt{-\frac {2}{\pi}\log\left(\sin\left(\frac{\theta n}{2}\right)\right)}.
\end{equation}
Then we set $Arg(\alpha_{x_0})=0$ and we define $Arg(\alpha_{x_{n+1}})\in [0,2\pi)$ to be determined by 
\begin{equation}\label{argcoef}
\left\{\begin{array}{c}
\cos(\tau_{n+1})=-\cos(\phi_{|\alpha_{x_n}|}-\phi_{|\alpha_{x_{n+1}}|}+\beta_n+Arg(\alpha_{x_n})-Arg(\alpha_{x_{n+1}})) ,\\ 
\delta_{n+1}=-sgn(\sin(\phi_{|\alpha_{x_n}|}-\phi_{|\alpha_{x_{n+1}}|}+\beta_n+Arg(\alpha_{x_n})-Arg(\alpha_{x_{n+1}}))),
\end{array}\right.
\end{equation}
where $\{\phi_{|\alpha_{x_n}|}\}$ are defined in Lemma \ref{Avect} and depend only on $|\alpha_{x_n}|$, and
 $$\beta_n=(|\alpha_{x_n}|^2-|\alpha_{x_{n+1}}|^2)\log|x_n-x_{n+1}|.$$

We consider the solution $u(t,x)$ given by Theorem \eqref{thdiracs} for the sequence $\sqrt{4\pi i} \alpha_k$ and $\frac 12<\gamma<1$, that solves 
\begin{equation}\label{cubicmodfinal}\begin{array}{c}i\partial_t u +\Delta u+\frac 12\left(|u|^2-\frac{2M}{t}\right)u=0,
\end{array}
\end{equation}
with $M=\sum_{k\in\mathbb Z}|\alpha_k|^2$, and can be written as
\begin{equation}\label{ansatzfinal}u(t,x)=\sum_{k\in\mathbb Z}e^{-i|\alpha_k|^2\log \sqrt{t}}(\alpha_k+R_k(t))\frac{e^{i\frac{|x-k|^2}{4t}}}{\sqrt{t}},\end{equation}
 with 
$$\sup_{0<t<T}t^{-\gamma}\|\{R_k(t)\}\|_{l^{2,s}}+t\,\|\{\partial_t R_k(t)\}\|_{l^{2,s}}<C.$$

\subsection{Construction of a solution of the binormal flow for $t>0$}\label{sectconstr}
Given an orthonormal basis $(v_1, v_2, v_3)$ of $\mathbb R^3$, 
a point $P\in\mathbb R^3$ and a time $t_0>0$ we construct a frame at all points $x\in\mathbb R$ and times $t>0$ by imposing{\footnote{Actually we should work in the definition of the evolution in time and in space laws for the frame with $v(t,x)=e^{iM\log\sqrt{t}}u(t,x)$ instead of $u(t,x)$. Indeed, this construction leads, by identifying $T_{tx}=T_{xt},e_{1tx}=e_{1xt},e_{2tx}=e_{2xt}$ to the NLS equation \eqref{cubicmodfinal} with nonlinearity $\frac 12\left(|v|^2-\frac{M}{t}\right)v$. However, for simplicity of the presentation we shall use the notation $u(t,x)$.  }} $(T,e_1,e_2)(t_0,0)=(v_1,v_2,v_3)$,
$$\left(\begin{array}{c}
T\\ e_1\\ e_2
\end{array}\right)_t(t,x)=
\left(\begin{array}{ccc}
0 & -\Im u_x & \Re u_x \\ \Im u_x & 0 & -\frac{|u|^2}{2}+\frac{M}{2t}\\ -\Re u_x &  \frac{|u|^2}{2}-\frac{M}{ 2t} & 0 
\end{array}\right)
\left(\begin{array}{c}
T\\ e_1\\ e_2
\end{array}\right)(t,x),$$
and
$$\left(\begin{array}{c}
T\\ e_1\\ e_2
\end{array}\right)_x(t,x)=
\left(\begin{array}{ccc}
0 & \Re u & \Im u \\ -\Re u & 0 & 0 \\ -\Im u &  0 & 0 
\end{array}\right)
\left(\begin{array}{c}
T\\ e_1\\ e_2
\end{array}\right)(t,x).$$
We can already notice that $T(t,x)$ satisfies the Schr\"odinger map: 
$$T_t=T\wedge T_{xx}.$$
We define now for all points $x\in\mathbb R$ and times $t>0$:
$$\chi(t,x)=P+\int_t^{t_0}(T\wedge T_{x})(\tau,0)d\tau+\int_0^{x}T(t,s)ds.$$
As $T(t,x)$ satisfies the Schr\"odinger map we have $T_t=(T\wedge T_x)_x$, so we can easily compute that $\chi(t,x)$ satisfies the binormal flow: 
$$\chi_t=T\wedge T_x=\chi_x\wedge\chi_{xx}.$$
We note that there are two degrees of freedom in this construction - the choice of the orthonormal basis $(v_1, v_2, v_3)$ of $\mathbb R^3$ and the choice of the point $P\in\mathbb R^3$. Changing these parameters is equivalent to rotate and translate respectively the solution $\chi(t)$.  The resulting evolution of curves is still a solution of the binormal flow, with the same laws of evolution in time and space for the frame. 
 
We introduce now the complex valued normal vector
$$N(t,x)=e_1(t,x)+ie_2(t,x).$$
With this vector we can write in a simpler way the laws of evolution in time and space for the frame, that will be useful in the rest of the proof:
\begin{equation}
\label{Tx}
T_x=\Re u \,e_1+\Im u \,e_2=\Re(\overline u\, N),
\end{equation}

\begin{equation}
\label{Nx}N_x=e_{1x}+ie_{2x}=-\Re u\, T-i\Im u \,T=-u\, T,
\end{equation}

\begin{equation}
\label{Tt}T_t=-\Im u_x\, e_1+\Re u_x\,e_2=\Im (\overline{u_x}\,N),
\end{equation}

\begin{equation}
\label{Nt}N_t=\Im u_x\,T+\left(-\frac{|u|^2}{2}+\frac{M}{2t}\right) e_2-i\Re u_x\,T+i\left(\frac{|u|^2}{2}-\frac{M}{2t}\right) e_1=-iu_x\, T+i\left(\frac{|u|^2}{2}-\frac{M}{2t}\right) N,
\end{equation}

\begin{equation}
\label{chit}\chi_t=T\wedge T_x=T\wedge\Re(\overline{u}\,N)=\Im(\overline{u}\, N).
\end{equation}
In particular from \eqref{ansatzfinal} and \eqref{chit} we have for $0<t_1<t_2<1$:
$$|\chi(t_2,x)-\chi(t_1,x)|=\left|\int_{t_1}^{t_2}\chi_t(t,x)dt\right|\leq\int_{t_1}^{t_2}|u(t,x)|dt$$
$$\leq \int_{t_1}^{t_2}\sum_j|\alpha_j+R_j(t)|\frac{dt}{\sqrt{t}}\leq C\sqrt{t_2}(\|\{\alpha_j\}\|_{l^{1}}+\|\{R_j\}\|_{L^\infty(0,1)l^{1}}).$$
This implies the existence of a limit curve at $t=0$ for all $x\in\mathbb R$:
$$\exists\, \underset{t\rightarrow 0}{\lim}\,\chi(t,x)=:\chi(0,x).$$
Moreover, $\chi(0)$ is a continuous curve. 

\subsection{Existence of a trace at $t=0$ for the tangent vector}\label{sectTlimit}
For further purposes we shall need to show that the tangent vector $T(t,x)$ has a limit $T(0,x)$ at $t=0$, and moreover we shall need to get a convergence decay of selfsimilar type  $\frac{\sqrt{t}}{d(x,\mathbb Z)}$ for $x$ close to $\mathbb Z$. This is insured by the following lemma.
\begin{lemma}\label{lemmaconvT}
Let $0<t_1<t_2<1$. If $x\in\mathbb R\backslash \frac 12\mathbb Z$ then
\begin{equation}\label{convT}
|T(t_2,x)-T(t_1,x)|\leq C(1+|x|)\sqrt{t_2}\left(\frac 1{d(x,\frac 12\mathbb Z)}+\frac1{d(x,\mathbb Z)}\right),
\end{equation}
while if $x\in \frac 12\mathbb Z$ then
\begin{equation}\label{convTZ}
|T(t_2,x)-T(t_1,x)|\leq C(1+|x|)\sqrt{t_2}.
\end{equation}
In particular for any $x\in\mathbb R$ there exists a trace for the tangent vector at $t=0$:
\begin{equation}\label{04}\exists\, \underset{t\rightarrow 0}{\lim}\,T(t,x)=:T(0,x).
\end{equation}
\end{lemma}

\begin{proof}
In view of \eqref{Tt} and \eqref{ansatzfinal} we have
$$T(t_2,x)-T(t_1,x)=\int_{t_1}^{t_2} \Im(\overline{u_x}\,N(t,x)) \,dt$$
$$=\Im\int_{t_1}^{t_2}e^{-iM\log\sqrt{t}}\sum_je^{i|\alpha_j|^2\log \sqrt{t}}(\overline{\alpha_j+R_j(t)})\frac{e^{-i\frac{(x-j)^2}{4t}}}{\sqrt{t}}(-i)\frac{(x-j)}{2t}N(t,x)\,dt.$$

In case $j=x$ the integrant vanishes so we get the left-hand-side of \eqref{convTZ} null. 

For $j\neq x$ we perform an integration by parts that exploits the oscillatory phase to get integrability in time:
$$T(t_2,x)-T(t_1,x)=\left[\Im e^{-iM\log\sqrt{t}}\sum_{j\neq x}e^{i|\alpha_j|^2\log \sqrt{t}}(\overline{\alpha_j+R_j(t)})\frac{e^{-i\frac{(x-j)^2}{4t}}}{\sqrt{t}}(-i)\frac{4t^2}{(x-j)^2}(-i)\frac{(x-j)}{2t}N(t,x)\right]_{t_1}^{t_2}$$
$$+2\Im\int_{t_1}^{t_2}e^{-iM\log\sqrt{t}}\sum_{j\neq x}\frac{e^{-i\frac{(x-j)^2}{4t}}}{x-j}(\sqrt{t}\,e^{i|\alpha_j|^2\log \sqrt{t}}(\overline{\alpha_j+R_j(t)})N(t,x))_t\,dt=:I_0+I_1+I_2+I_3+I_4,$$
where we have denoted by $I_0$ the boundary term and by $I_1,I_2,I_3,I_4$ the terms obtained after the differentiation in time of the quadruple product in the integral part. 
We consider first the boundary term 
$$|I_0|\leq C\sqrt{t_2}\sum_{j\neq x}|\alpha_j+R_j(t_2)|\frac{1}{|x-j|}+C\sqrt{t_1}\sum_{j\neq x}|\alpha_j+R_j(t_1)|\frac{1}{|x-j|}.$$
If $x\in\mathbb Z$ then we have
$$|I_0|\leq C\sqrt{t_2}(\|\{\alpha_j\}\|_{l^{1}}+\|\{R_j\}\|_{L^\infty(0,t_2)l^{1}}),$$
while if $x\notin\mathbb Z$ 
$$|I_0|\leq C\frac{\sqrt{t_2}}{d(x,\mathbb Z)}(\|\{\alpha_j\}\|_{l^{1}}+\|\{R_j\}\|_{L^\infty(0,t_2)l^{1}}).$$
Therefore the contribution of $I_0$ fits with the estimates in the statement of the Lemma. The terms $I_1$ and $I_2$ can be treated the same, as $\int_{t_1}^{t_2}(\sqrt{t}e^{-i|\alpha_k|^2\log \sqrt{t}})_t\,dt\leq C\sqrt{t_2}$. Also the term $I_3$ can be treated similarly, as $|\partial_t R_j(t))|\leq \frac Ct$ on $(0,1)$. We are left with the $I_4$ term, which contains in view of \eqref{Nt} the expression $N_t=-iu_x\, T+i\left(\frac{|u|^2}{2}-\frac{M}{2t}\right) N$:
$$I_4=2\Im\int_{t_1}^{t_2}e^{-iM\log\sqrt{t}}\sum_{j\neq x}\frac{e^{-i\frac{(x-j)^2}{4t}}}{x-j}\sqrt{t}\,e^{i|\alpha_j|^2\log \sqrt{t}}(\overline{\alpha_j+R_j(t)})$$
$$\times \left(-ie^{iM\log\sqrt{t}}\sum_ke^{-i|\alpha_k|^2\log \sqrt{t}}(\alpha_k+R_k(t))\,\frac{e^{i\frac{(x-k)^2}{4t}}}{\sqrt{t}}i\frac{(x-k)}{2t}T(t,x)\right.$$
$$\left.+i\left(\frac{\sum_{r,\tilde r}e^{-i(|\alpha_r|^2-|\alpha_{\tilde r}|^2)\log \sqrt{t}}(\alpha_r+R_r(t))(\overline{\alpha_{\tilde r}+R_{\tilde r}(t)})e^{i\frac{(x-r)^2-(x-\tilde r)^2}{4t}}}{2 t}-\frac M{2t}\right)N(t,x)\right)\,dt.$$
We notice that the second term can be upper-bounded as $I_0$. We are thus left with the first term:
$$I_{4,1}=\Im\int_{t_1}^{t_2}\sum_{j,k\neq x}e^{-i\frac{(j-k)(j+k-2x)}{4t}}\frac{x-k}{x-j}e^{i(|\alpha_j|^2-|\alpha_k|^2)\log \sqrt{t}}(\overline{\alpha_j+R_j(t)})(\alpha_k+R_k(t))T(t,x)\frac{dt}{t},$$
for which we still have an obstruction for the integrability in time. 
The terms in the sum for which $j=k$ have null contribution as they are real numbers. Also, in case $2x\in\mathbb Z$, the terms in the sum for which $j+k-2x=0$ give
$$-\Im\int_{t_1}^{t_2}\sum_{k\neq x}e^{i(|\alpha_{-k+2x}|^2-|\alpha_k|^2)\log \sqrt{t}}(\overline{\alpha_{-k+2x}+R_{-k+2x}(t)})(\alpha_k+R_k(t))T(t,x)\frac{dt}{t}$$
$$=-\Im\int_{t_1}^{t_2}\sum_{j\neq x}e^{i(|\alpha_j|^2-|\alpha_{-j+2x}|^2)\log \sqrt{t}}(\overline{\alpha_{j}+R_{j}(t)})(\alpha_{-j+2x}+R_{-j+2x}(t))T(t,x)\frac{dt}{t}$$
$$=\Im\int_{t_1}^{t_2}\sum_{j\neq x}e^{i(-|\alpha_j|^2+|\alpha_{-j+2x}|^2)\log \sqrt{t}}(\alpha_{j}+R_{j}(t))(\overline{\alpha_{-j+2x}+R_{-j+2x}(t)})T(t,x)\frac{dt}{t}$$
$$=\Im\int_{t_1}^{t_2}\sum_{k\neq x}e^{i(|\alpha_{-k+2x}|^2-|\alpha_k|^2)\log \sqrt{t}}(\overline{\alpha_{-k+2x}+R_{-k+2x}(t)})(\alpha_k+R_k(t))T(t,x)\frac{dt}{t},$$
so their contribution is null.\\
 Therefore we have, for all $x\in\mathbb R$,
$$I_{4,1}=\Im\int_{t_1}^{t_2}\sum_{j,k\neq x;\,j\neq k;\,j+k\neq 2x}e^{-i\frac{(j-k)(j+k-2x)}{4t}}\frac{x-k}{x-j}$$
$$\times e^{i(|\alpha_j|^2-|\alpha_k|^2)\log \sqrt{t}}(\overline{\alpha_j+R_j(t)})(\alpha_k+R_k(t))T(t,x)\frac{dt}{t}.$$
We perform an integration by parts:
$$I_{4,1}=\Im\left[\sum_{j,k\neq x;\,j\neq k;\,j+k\neq 2x}e^{-i\frac{(j-k)(j+k-2x)}{4t}}\frac{(-i)4t^2}{(j-k)(j+k-2x)}\frac{x-k}{x-j}\right.$$
$$\left.\times e^{i(|\alpha_j|^2-|\alpha_k|^2)\log \sqrt{t}}(\overline{\alpha_j+R_j(t)})(\alpha_k+R_k(t))\frac{T(t,x)}{t}\right]_{t_1}^{t_2}$$
$$+4\Im\int_{t_1}^{t_2}\sum_{j,k\neq x;\,j\neq k;\,j+k\neq 2x}e^{-i\frac{(j-k)(j+k-2x)}{4t}}\frac{i}{(j-k)(j+k-2x)}\frac{x-k}{x-j}$$
$$\times (te^{i(|\alpha_j|^2-|\alpha_k|^2)\log \sqrt{t}}(\overline{\alpha_j+R_j(t)})(\alpha_k+R_k(t))T(t,x))_t\,dt$$
$$=:I_{4,1}^0+I_{4,1}^1+I_{4,1}^2+I_{4,1}^3+I_{4,1}^4+I_{4,1}^5,$$
where $I_{4,1}^0$ stands for the boundary term and $I_{4,1}^1,I_{4,1}^2,I_{4,1}^3,I_{4,1}^4$ and $I_{4,1}^5$ are the terms after differentiating in time the quintuple product in the integral. For the boundary term we have
$$|I_{4,1}^0|\leq 4t_2\sum_{j,k\neq x;\,j\neq k;\,j+k\neq 2x}\frac{|x-k|}{|j-k||j+k-2x||x-j|}|\alpha_j+R_j(t_2)||\alpha_k+R_k(t_2)|$$
$$+4t_1\sum_{j,k\neq x;\,j\neq k;\,j+k\neq 2x}\frac{|x-k|}{|j-k||j+k-2x||x-j|}|\alpha_j+R_j(t_1)||\alpha_k+R_k(t_1)|.$$
As for $j\neq k$
\begin{equation}\label{estsum1}\frac{|x-k|}{|j-k||j+k-2x||x-j|}\leq \frac{|x-j|+|j+k-2x|}{|j-k||j+k-2x||x-j|}\leq\frac1{|j+k-2x|}+\frac1{|x-j|},\end{equation}
we have for $x\in\frac 12\mathbb Z$ 
$$|I_{4,1}^0|\leq Ct_2(\|\{\alpha_j\}\|_{l^{1}}+\|\{R_j\}\|_{L^\infty(0,t_2)l^{1}})^2.$$
while for $x\notin \frac 12\mathbb Z$ we obtain
$$|I_{4,1}^0|\leq Ct_2\left(\frac 1{d(x,\frac 12\mathbb Z)}+\frac1{d(x,\mathbb Z)}\right)(\|\{\alpha_j\}\|_{l^{1}}+\|\{R_j\}\|_{L^\infty(0,t_2)l^{1}})^2.$$
The terms $I_{4,1}^1,I_{4,1}^2,I_{4,1}^3$ and $I_{4,1}^4$ can be upper-bounded as $I_{4,1}^0$ by using moreover for $I_{4,1}^3$ and $I_{4,1}^4$ the bound $\partial_t R_j(t))\leq \frac Ct$ on $(0,1)$. The last term $I_{4,1}^5$ involves, in view of \eqref{Tt}, 
$$T_t(t,x)=\Im (\overline{u_x}\,N)(t,x)=\Im e^{-iM\log\sqrt{t}}\sum_re^{i|\alpha_r|^2\log \sqrt{t}}(\overline{\alpha_r+R_r(t)})\frac{e^{-i\frac{(x-r)^2}{4t}}}{\sqrt{t}}(-i)\frac{(x-r)}{2t}N(t,x)$$ 
so 
$$I_{4,1}^5=-\frac1{2}\Im\int_{t_1}^{t_2}\sum_{j,k\neq x;\,j\neq k;\,j+k\neq 2x}e^{-i\frac{(j-k)(j+k-2x)}{4t}}\frac{i}{(j-k)(j+k-2x)}\frac{x-k}{x-j}e^{i(|\alpha_j|^2-|\alpha_k|^2)\log \sqrt{t}}$$
$$\times (\overline{\alpha_j+R_j(t)})(\alpha_k+R_k(t))\Re e^{-iM\log\sqrt{t}}\sum_re^{i|\alpha_r|^2\log \sqrt{t}}(\overline{\alpha_r+R_r(t)})e^{-i\frac{(x-r)^2}{4t}}(x-r)N(t,x)\,\frac{dt}{\sqrt{t}},$$
and in particular
$$|I_{4,1}^5|\leq C\int_{t_1}^{t_2}\sum_{j,k\neq x;\,j\neq k;\,j+k\neq 2x}\frac{|x-k|}{|j-k||j+k-2x||x-j|}$$
$$\times |\alpha_j+R_j(t)||\alpha_k+R_k(t)|\sum_r|\alpha_r+R_r(t)||x-r|\,\frac{dt}{\sqrt{t}}.$$
We can write
$$\sum_r|\alpha_r+R_r(t)||x-r|\leq  C(1+|x|)(\|\{\alpha_j\}\|_{l^{2,\frac 32^+}}+\|\{R_j\}\|_{L^\infty(0,t_2)l^{2,\frac 32^+}}),$$
so by using \eqref{estsum1} we get for $x\in\frac 12\mathbb Z$:
$$|I_{4,1}^5|\leq C(1+|x|)\sqrt{t_2}(\|\{\alpha_j\}\|_{l^{1}}+\|\{R_j\}\|_{L^\infty(0,t_2)l^{1}})^2(\|\{\alpha_j\}\|_{l^{2,\frac 32^+}}+\|\{R_j\}\|_{L^\infty(0,t_2)l^{2,\frac 32^+}}),$$
while for $x\notin \frac 12\mathbb Z$ we obtain:
$$|I_{4,1}^5|\leq C\sqrt{t_2}(1+|x|)\left(\frac 1{d(x,\frac 12\mathbb Z)}+\frac1{d(x,\mathbb Z)}\right)$$
$$\times (\|\{\alpha_j\}\|_{l^{1}}+\|\{R_j\}\|_{L^\infty(0,t_2)l^{1}})^2(\|\{\alpha_j\}\|_{l^{2,\frac 32^+}}+\|\{R_j\}\|_{L^\infty(0,t_2)l^{2,\frac 32^+}}).$$
Therefore the proof of the Lemma is completed. 
\end{proof}

\subsection{Segments of the limit curve at $t=0$}\label{sectsegm}
\begin{lemma}\label{lemmasegm}
Let $n\in\mathbb Z$ and $x_1,x_2\in (n,n+1)$. Then
$$T(0,x_1)=T(0,x_2).$$
In particular, we recover that $\chi(0)$ is a polygonal line, and might have corners only at integer locations.
\end{lemma}

\begin{proof}
From Lemma \ref{lemmaconvT} we have
\begin{equation}\label{convTdiff}T(0,x_1)-T(0,x_2)=\underset{t\rightarrow 0}{\lim} \,(T(t,x_1)-T(t,x_2)).
\end{equation}
In view of \eqref{Tx} we compute
$$T(t,x_1)-T(t,x_2)=\int_{x_1}^{x_2}\Re(\overline{u}N(t,x))\,dx$$
$$=\Re e^{-iM\log\sqrt{t}}\int_{x_1}^{x_2}\sum_je^{i|\alpha_j|^2\log \sqrt{t}}(\overline{\alpha_j+R_j(t)})\frac{e^{-i\frac{(x-j)^2}{4t}}}{\sqrt{t}}\,N(t,x)\,dx.$$
In this case the integral is well defined, but we need decay in time. 
For this purpose we perform an integration by parts, that is allowed on $(x_1,x_2)\subset(n,n+1)$:
$$T(t,x_1)-T(t,x_2)=\left[\Re e^{-iM\log\sqrt{t}}\sum_{j}e^{i|\alpha_j|^2\log \sqrt{t}}(\overline{\alpha_j+R_j(t)})\frac{e^{-i\frac{(x-j)^2}{4t}}}{\sqrt{t}}i\frac{2t}{x-j}\,N(t,x)\right]_{x_1}^{x_2}$$
$$+2\sqrt{t}\Im e^{-iM\log\sqrt{t}}\int_{x_1}^{x_2}\sum_{j}e^{i|\alpha_j|^2\log \sqrt{t}}(\overline{\alpha_j+R_j(t)})e^{-i\frac{(x-j)^2}{4t}}\left(\frac 1{x-j}\,N(t,x)\right)_x\,dx$$
$$=O(\sqrt{t})+2\sqrt{t}\Im e^{-iM\log\sqrt{t}}\int_{x_1}^{x_2}\sum_{j}e^{i|\alpha_j|^2\log \sqrt{t}}(\overline{\alpha_j+R_j(t)})e^{-i\frac{(x-j)^2}{4t}}\frac 1{x-j}\,N_x(t,x)\,dx.$$
As by \eqref{Nx} we have $N_x=-uT$,
$$T(t,x_1)-T(t,x_2)=O(\sqrt{t})$$
$$-2\Im\sum_{j,k}e^{i(|\alpha_j|^2-|\alpha_k|^2)\log \sqrt{t}}(\overline{\alpha_j+R_j(t)})(\alpha_k+R_k(t))\int_{x_1}^{x_2}\frac{e^{-i\frac{(x-j)^2-(x-k)^2}{4t}}}{x-j}T(t,x)\,dx.$$
The summation holds only for $j\neq k$, as for $j=k$ the contribution is null. Moreover, from \eqref{decayansatzcubic} we have $\|\{R_j(t)\}\|_{l^1}=O(t^\gamma)$, $\gamma>1/2$, so 
$$T(t,x_1)-T(t,x_2)=O(\sqrt{t})-2\Im\sum_{j\neq k}e^{i(|\alpha_j|^2-|\alpha_k|^2)\log \sqrt{t}}\overline{\alpha_j}\alpha_ke^{i\frac{j^2-k^2}{4t}}\int_{x_1}^{x_2}\frac{e^{i\frac{(j-k)x}{2t}}}{x-j}T(t,x)\,dx.$$
To get decay in time we need to perform again an integration by parts:
$$T(t,x_1)-T(t,x_2)=O(\sqrt{t})-\left[2\Im\sum_{j\neq k}e^{i(|\alpha_j|^2-|\alpha_k|^2)\log \sqrt{t}}\overline{\alpha_j}\alpha_ke^{i\frac{j^2-k^2}{4t}}\frac{e^{i\frac{(j-k)x}{2t}}}{x-j}\frac{2t}{i(j-k)}T(t,x)\right]_{x_1}^{x_2}$$
$$+4t\,\Re\sum_{j\neq k}e^{i(|\alpha_j|^2-|\alpha_k|^2)\log \sqrt{t}}\overline{\alpha_j}\alpha_k\frac{e^{i\frac{j^2-k^2}{4t}}}{j-k}\int_{x_1}^{x_2}e^{i\frac{(j-k)x}{2t}}\left(\frac 1{x-j}T(t,x)\right)_x\,dx$$
$$=O(\sqrt{t})+4t\,\Re\sum_{j\neq k}e^{i(|\alpha_j|^2-|\alpha_k|^2)\log \sqrt{t}}\overline{\alpha_j}\alpha_k\frac{e^{i\frac{j^2-k^2}{4t}}}{j-k}\int_{x_1}^{x_2}e^{i\frac{(j-k)x}{2t}}\frac 1{x-j}T_x(t,x)\,dx.$$
From \eqref{Tx} we have $T_x=\Re(\overline{u}\,N)$ so finally
$$T(t,x_1)-T(t,x_2)=O(\sqrt{t})+4t\,\Re\sum_{j\neq k}e^{i(|\alpha_j|^2-|\alpha_k|^2)\log \sqrt{t}}\overline{\alpha_j}\alpha_k\frac{e^{i\frac{j^2-k^2}{4t}}}{j-k}$$
$$\times \int_{x_1}^{x_2}\frac{e^{i\frac{(j-k)x}{2t}}}{x-j}\Re e^{-iM\log\sqrt{t}}\left(\sum_re^{i|\alpha_r|^2\log \sqrt{t}}(\overline{\alpha_r+R_r(t)})\frac{e^{-i\frac{(x-r)^2}{4t}}}{\sqrt{ t}}\,N(t,x)\right)\,dx=O(\sqrt{t}).$$
Therefore in view of \eqref{convTdiff} we have indeed
$$T(0,x_1)-T(0,x_2)=0.$$
\end{proof}

\subsection{Recovering self-similar structures through self-similar paths}\label{sectcorners}
In this subsection we shall use the results in \cite{GRV} that characterize all the selfsimilar solutions of BF and give their corresponding asymptotics (see Theorem 1 in \cite{GRV}). 

Let us denote by $A^\pm_{|\alpha_k|}\in\mathbb S^2$ the directions of the corner generated at time $t=0$ by the canonical self-similar solution $\chi_{|\alpha_k|}(t,x)$ of the binormal flow of curvature $\frac{|\alpha_k|}{\sqrt{t}}$:
$$A^\pm_{|\alpha_k|}:=\partial_x\chi_{|\alpha_k|}(0,0^\pm).$$
We recall also that the frame of the profile (i.e. $\chi_{|\alpha_k|}(1)$) satisfies the system
\begin{equation}\label{systselfs}
\left\{\begin{array}{c}\partial_xT_{|\alpha_k|}(1,x)=\Re(|\alpha_k|e^{-i\frac{x^2}{4}}N_{|\alpha_k|}(1,x)),\\\partial_xN_{|\alpha_k|}(1,x)=-|\alpha_k|e^{i\frac{x^2}{4}}T_{|\alpha_k|}(1,x),\end{array}\right.\end{equation}
and that for $x\rightarrow\pm\infty$ there exist $B^\pm_{|\alpha_k|}\perp A^\pm_{|\alpha_k|}$, with $\Re(B^\pm_{|\alpha_k|}),\Im(B^\pm_{|\alpha_k|})\in\mathbb S^2$, such that
\begin{equation}\label{asselfs}T_{|\alpha_k|}(1,x)=A^\pm_{|\alpha_k|}+\mathcal O(\frac 1x),\quad e^{i|\alpha_k|^2\log |x|}N_{|\alpha_k|}(1,x)=B^\pm_{|\alpha_k|}+\mathcal O(\frac 1{|x|}).\end{equation}

\begin{lemma}\label{lemmassstructure}
Let $t_n$ be a sequence of positive times converging to zero.  
Up to a subsequence, there exists for all $x\in\mathbb R$ a limit
$$(T_*(x),N_*(x))=\underset{n\rightarrow \infty}{\lim}\,(T(t_n,k+x\sqrt{t_n}),e^{-iM\log\sqrt{t_n}}e^{i|\alpha_k|^2\log \sqrt{t_n}}N(t_n,k+x\sqrt{t_n})),$$
and  there exists a unique rotation $\Theta_k$ such that
\begin{equation}\label{ssstructure}\left\{\begin{array}{c}T_*(x)=\Theta_k(T_{|\alpha_k|}(x)),\\ \Re(e^{i Arg\alpha_k}N_*(x))=\Theta_k(\Re(N_{|\alpha_k|}(x))),\\ \Im(e^{i Arg\alpha_k}N_*(x))=\Theta_k(\Im(N_{|\alpha_k|}(x))).\end{array}\right.\end{equation}
Moreover, for $x\rightarrow\pm\infty$
\begin{equation}\label{ssstructurelim}T_*(x)=\Theta_k(A^\pm_{|\alpha_k|})+\mathcal O(\frac 1{|x|}), \quad e^{i|\alpha_k|^2\log|x|}e^{i Arg(\alpha_k)}N_*(x)=\Theta_k(B^\pm_{|\alpha_k|})+\mathcal O(\frac 1{|x|}).\end{equation}
\end{lemma}

\begin{proof}
Let $t_n$ be a sequence of positive times converging to zero. We introduce for $x\in\mathbb R$ the functions
$$(T_n(x),N_n(x))=(T(t_n,k+x\sqrt{t_n}),e^{-iM\log\sqrt{t_n}}e^{i|\alpha_k|^2\log \sqrt{t_n}}N(t_n,k+x\sqrt{t_n})).$$
This sequence is uniformly bounded. In view of \eqref{Tx} and \eqref{Nx} we have
$$T_n'(x)=\sqrt{t_n}\Re(\overline{u}N)(t_n,k+x\sqrt{t_n})$$
$$=\Re e^{-iM\log\sqrt{t_n}}\sum_je^{i|\alpha_j|^2\log \sqrt{t_n}}(\overline{\alpha_j+R_j(t_n))}e^{-i\frac{(k+x\sqrt{t_n}-j)^2}{4t_n}}\, N(t_n,k+x\sqrt{t_n})), $$
and
$$N_n'(x)=-e^{-iM\log\sqrt{t_n}}e^{i|\alpha_k|^2 \log \sqrt{t_n}}\sqrt{t_n}(uT)(t_n,k+x\sqrt{t_n})$$
$$=-\sum_je^{i(|\alpha_k|^2-|\alpha_j|^2)\log \sqrt{t_n}}(\alpha_j+R_j(t_n))e^{i\frac{(k+x\sqrt{t_n}-j)^2}{4t_n}}\, T(t_n,k+x\sqrt{t_n}).$$
Therefore the sequence $(T_n'(x),N_n'(x))$ is also uniformly bounded. These two facts give via Arzela-Ascoli's theorem the existence of a limit in $n$ (of a subsequence, that we denote again $(T_n(x),N_n(x))$):
$$\exists\, \underset{n\rightarrow \infty}{\lim}\,(T_n(x),N_n(x))=:(T_*(x),N_*(x)).$$
Moreover, as $\|\{R_j(t_n)\}\|_{l^1}=o(n)$ we can write
$$T_n'(x)=\Re e^{-iM\log\sqrt{t_n}}\sum_je^{i|\alpha_j|^2\log \sqrt{t_n}}\overline{\alpha_j}e^{-i\frac{(k+x\sqrt{t_n}-j)^2}{4t_n}}\, N(t_n,k+x\sqrt{t_n}))+o(n)N_n(x)$$
$$=\Re(\overline{\alpha_k}e^{-i\frac{x^2}{4}}\, N_n(x))+\Re(f_n(x)N_n(x))+o(n)N_n(x),$$
where
$$f_n(x)=\sum_{j\neq k}e^{i(|\alpha_j|^2-|\alpha_k|^2)\log \sqrt{t_n}}\overline{\alpha_j}e^{-i\frac{x^2}{4}+ix\frac{j-k}{2\sqrt{t_n}}-i\frac{(k-j)^2}{4t_n}}.$$
For a test function $\psi\in\mathcal C^\infty_c(\mathbb R)$ we have by integrating by parts, avoiding in case a region os size $o(n)$ around $x=0$,
$$\langle f_n(x),\psi(x)\rangle=\int \sum_{j\neq k}e^{i(|\alpha_j|^2-|\alpha_k|^2)\log \sqrt{t_n}}\overline{\alpha_j}e^{-i\frac{x^2}{4}+ix\frac{j-k}{2\sqrt{t_n}}-i\frac{(k-j)^2}{4t_n}}\psi(x)\,dx$$
$$=-\int \sum_{j\neq k}e^{i(|\alpha_j|^2-|\alpha_k|^2)\log \sqrt{t_n}}\overline{\alpha_j}2\sqrt{t_n}\frac{e^{ix\frac{j-k}{2\sqrt{t_n}}-i\frac{(k-j)^2}{4t_n}}}{i(j-k)}(e^{-i\frac{x^2}{4}}\psi(x))_x\,dx=C(\psi)o(n).$$
Similarly we obtain
$$N_n'(x)=-\alpha_ke^{i\frac{x^2}{4}}\, T_n(x)+g_n(x)T_n(x)+o(n)T_n(x),$$
with $g_n=o(n)$ in the weak sense. 
Therefore $(T_*(x),e^{i Arg(\alpha_k)}N_*(x))$ satisfies system \eqref{systselfs} in the weak sense. As the coefficients involved in the ODE are analytic we conclude that $(T_*(x),e^{i Arg(\alpha_k)}N_*(x))$ satisfies system \eqref{systselfs} in the strong sense, as $(T_{|\alpha_k|}(x),N_{|\alpha_k|}(x))$ does. Therefore there exists a unique rotation $\Theta_k$ such that \eqref{ssstructure} holds. 
We obtain \eqref{ssstructurelim} as a consequence of \eqref{asselfs}.
\end{proof}

\subsection{Recovering the curvature angles of the initial data}\label{sectcorners}
\begin{lemma}\label{lemmaangles}
Let $k\in\mathbb Z$. Then, with the notations of the previous subsection,
$$T(0,k^\pm)=\Theta_k(A^\pm_{|\alpha_k|}).$$
In particular, in view of \eqref{angless} and \eqref{moduluscoef} we recover that $\chi(0)$ is a polygonal line with corners at the same locations as $\chi_0$, and of same angles.\footnote{ This also implies that the rotation $\Theta_k$ does not depend on the choice of the sequence $t_n$. }
\end{lemma}
\begin{proof}
Let $\epsilon>0$. In view of \eqref{ssstructurelim} we first choose $x>0$ large enough such that
$$|T_*(x)-\Theta_k(A^+_{|\alpha_k|}|)\leq \frac\epsilon 3,$$
and that $\frac{C(1+|k|)}{x}\leq\frac\epsilon 3$, where $C$ is the coefficient in \eqref{convT}.  
Then we choose $n$ large enough such that $|x\sqrt{t_n}|<\frac 12$ and that $|T(t_n,k+x\sqrt{t_n})-T_*(x)|\leq\frac \epsilon 3$. The last fact is possible in view of Lemma \ref{lemmassstructure}. Finally, we have, in view of Lemma \ref{lemmasegm} and \eqref{convT}:
$$|T(0,k^+)-\Theta_k(A^+_{|\alpha_k|})|=|T(0,k+x\sqrt{t_n})-\Theta_k(A^+_{|\alpha_k|})|$$
$$\leq |T(0,k+x\sqrt{t_n})-T(t_n,k+x\sqrt{t_n})|+|T(t_n,k+x\sqrt{t_n})-T_*(x)|+|T_*(x)-\Theta_k(A^+_{|\alpha_k|})|\leq\epsilon,$$
so
$$T(0,k^+)=\Theta_k(A^+_{|\alpha_k|}).$$
Similarly we show by taking $x<0$ that
$$T(0,k^-)=\Theta_k(A^-_{|\alpha_k|}).$$
\end{proof}
The lemma insures us that $\chi(0)$ has corners at the same locations as $\chi_0$, and of same angles. To recover $\chi_0$ up to rotation and translations we need to recover also the torsion properties of $\chi_0$.

\subsection{Trace and properties of modulated normal vectors}\label{sectN}
In order to recover the torsion angles we shall need to get informations about $N(t,x)$ as $t$ goes to zero. 
For $x\notin\mathbb Z$ we denote the modulated normal vector
\begin{equation}\label{Ntilde}
\tilde N(t,x)=e^{i\Phi(t,x)}N(t,x),
\end{equation}
where
\begin{equation}\label{phase}
\Phi(t,x)=\sum_{j\neq x}|\alpha_j|^2\log\frac{|x-j|}{\sqrt t}.
\end{equation}

We start with a lemma insuring the existence of a limit for $\tilde N(t,x)$ at $t=0$, with a convergence decay of selfsimilar type  $\frac{\sqrt{t}}{d(x,\mathbb Z)}$ for $x$ close to $\mathbb Z$.

\begin{lemma}\label{lemmaconvN}
Let $0<t_1<t_2<1$. For $x\notin\frac 12\mathbb Z$ we have
\begin{equation}\label{convN}
|\tilde N(t_2,x)-\tilde N(t_1,x)|\leq C(1+|x|)\sqrt{t_2}\left(\frac 1{d(x,\frac 12\mathbb Z)}+\frac1{d(x,\mathbb Z)}\right),
\end{equation}
while if $x\in \frac 12\mathbb Z\backslash\mathbb Z$ then
\begin{equation}\label{convNZ}
|\tilde N(t_2,x)-\tilde N(t_1,x)|\leq C(1+|x|)\sqrt{t_2}.
\end{equation}
In particular for any $x\notin\mathbb Z$ there exists a trace for the modulated normal vector at $t=0$:
$$\exists\, \underset{t\rightarrow 0}{\lim}\,\tilde N(t,x)=:\tilde N(0,x).$$
Moreover for any $x\in\mathbb Z$ there exists a trace
\begin{equation}\label{03}\exists\, \underset{t\rightarrow 0}{\lim}\,e^{i\sum_{j\neq x}|\alpha_j|^2\log\frac{|x-j|}{\sqrt t}}N(t,x),
\end{equation} 
with a rate of convergence upper-bounded by $C(1+|x|)\sqrt{t}$.
\end{lemma}

\begin{proof}
In view of \eqref{Nt} and \eqref{ansatzfinal} we have
$$\tilde N(t_2,x)-\tilde N(t_1,x)=\int_{t_1}^{t_2}\left(-iu_x\, T+i\left(\frac{|u|^2}{2}-\frac{M}{2t}\right) N+i\Phi_t N\right)e^{i\Phi}dt$$
$$=\int_{t_1}^{t_2}\left(-ie^{iM\log\sqrt{t}}\sum_je^{-i|\alpha_j|^2\log \sqrt{t}}(\alpha_j+R_j(t))\frac{e^{i\frac{(x-j)^2}{4t}}}{\sqrt{t}}i\frac{(x-j)}{2t}T(t,x)\right.$$
$$\left.+i\sum_{j\neq k}e^{-i(|\alpha_j|^2-|\alpha_k|^2)\log \sqrt{t}}(\alpha_j+R_j(t))(\overline{\alpha_k+R_k(t)})\frac{e^{i\frac{(x-j)^2-(x-k)^2}{4t}}}{2t}N(t,x)+i\Phi_t N(t,x)\,\right)e^{i\Phi}dt.$$
We can integrate by parts in the first term to get
$$\tilde N(t_2,x)-\tilde N(t_1,x)=$$
$$=\left[e^{iM\log\sqrt{t}}\sum_{j\neq x}e^{-i|\alpha_j|^2\log \sqrt{t}}(\alpha_j+R_j(t))\frac{e^{i\frac{(x-j)^2}{4t}}}{\sqrt{t}}(-\frac{4t^2}{i(x-j)^2})\frac{(x-j)}{2t}T(t,x)\right]_{t_1}^{t_2}$$
$$-2i\int_{t_1}^{t_2}e^{iM\log\sqrt{t}}\sum_{j\neq x} \frac{e^{i\frac{(x-j)^2}{4t}}}{x-j} (\sqrt{t}e^{-i|\alpha_j|^2\log \sqrt{t}}(\alpha_j+R_j(t))T(t,x)e^{i\Phi})_tdt$$
$$+i\int_{t_1}^{t_2}(\sum_{j\neq k}e^{-i(|\alpha_j|^2-|\alpha_k|^2)\log \sqrt{t}}(\alpha_j+R_j(t))(\overline{\alpha_k+R_k(t)})\frac{e^{i\frac{(x-j)^2-(x-k)^2}{4t}}}{2t}N(t,x)+\Phi_t N(t,x))e^{i\Phi}dt.$$
Having in mind the expression \eqref{Tt} for $T_t$ we obtain
$$\tilde N(t_2,x)-\tilde N(t_1,x)=O(\frac{\sqrt{t_2}}{d(x,\mathbb Z)})$$
$$-2i\int_{t_1}^{t_2}e^{iM\log\sqrt{t}}\sum_{j\neq x} \frac{e^{i\frac{(x-j)^2}{4t}}}{x-j} \sqrt{t}e^{-i|\alpha_j|^2\log \sqrt{t}}(\alpha_j+R_j(t))$$
$$\times \Im(e^{-iM\log\sqrt{t}}\sum_ke^{i|\alpha_k|^2\log \sqrt{t}}(\overline{\alpha_k+R_k(t)})\frac{e^{-i\frac{(x-k)^2}{4t}}}{\sqrt{t}}(-i\frac{x-k}{2t})N(t,x) e^{i\Phi}dt$$
$$+i\int_{t_1}^{t_2}(\sum_{j\neq k}e^{-i(|\alpha_j|^2-|\alpha_k|^2)\log \sqrt{t}}(\alpha_j+R_j(t))(\overline{\alpha_k+R_k(t)})\frac{e^{i\frac{(x-j)^2-(x-k)^2}{4t}}}{2t}N(t,x)+\Phi_t N(t,x))e^{i\Phi}dt.$$
The integrals are in $\frac 1t$. 
By writing $\Im(-iz)=-\frac{z+\overline z}2$ in the first integral, we have terms $e^{i\frac{(x-j)^2-(x-k)^2}{4t}}$ or $e^{i\frac{(x-j)^2+(x-k)^2}{4t}}$, both oscillant except for the first one, in case $j=k$ or $2x=j+k$. For $x\notin\frac 12\mathbb Z$ we perform integrations by parts in all terms, except in case $j=k$ for the first integral, that allow for a gain of $t^2$ minus at worse terms involving $N_t$ that are in $\frac 1{t\sqrt{t}}$:
$$\tilde N(t_2,x)-\tilde N(t_1,x)=O((1+|x|)\sqrt{t_2}(\frac{1}{d(x,\frac 12\mathbb Z)}+\frac{1}{d(x,\mathbb Z)}))$$
$$+i\int_{t_1}^{t_2}\sum_{j\neq x}  \frac{|\alpha_j+R_j(t)|^2}{2t}N e^{i\Phi}+\Phi_t N(t,x)e^{i\Phi}dt.$$
In view of the decay of $\{R_j(t)\}$ and the expression \eqref{phase} of the phase $\Phi$ we obtain \eqref{convN}. 

We are left with the case $x\in\frac 12\mathbb Z$. The computations goes as above, with some extra non-oscillant terms that actually calcel:
 $$\tilde N(t_2,x)-\tilde N(t_1,x)=O((1+|x|)\sqrt{t_2}$$
 $$+i\int_{t_1}^{t_2}\sum_{j\neq x,k;j+k=2x}e^{-i(|\alpha_j|^2-|\alpha_k|^2)\log \sqrt{t}}(\alpha_j+R_j(t))(\overline{\alpha_k+R_k(t)})\frac{x-k}{x-j}\frac {N(t,x)}{2t} e^{i\Phi}dt$$
 $$+i\int_{t_1}^{t_2}\sum_{j\neq k;j+k=2x}e^{-i(|\alpha_j|^2-|\alpha_k|^2)\log \sqrt{t}}(\alpha_j+R_j(t))(\overline{\alpha_k+R_k(t)})\frac{N(t,x)}{2t}e^{i\Phi}dt.$$
$$+i\int_{t_1}^{t_2}\sum_{j\neq x}  \frac{|\alpha_j+R_j(t)|^2}{2t}N e^{i\Phi}+\Phi_t N(t,x)e^{i\Phi}dt=O((1+|x|)\sqrt{t_2}.$$
\end{proof}

Next we shall prove that $\tilde N(0,x)$ is piecewise constant. 
\begin{lemma}\label{lemmasegmN}
Let $n\in\mathbb Z$ and $x_1,x_2\in (n,n+1)$. Then
$$\tilde N(0,x_1)=\tilde N(0,x_2).$$
Moreover, the same statement remains valid for $x_1,x_2\in(n-1,n+1)$ if $\alpha_n=0$.
\end{lemma}

\begin{proof}
From Lemma \ref{lemmaconvN} we have
\begin{equation}\label{convNdiff}\tilde N(0,x_1)-\tilde N(0,x_2)=\underset{t\rightarrow 0}{\lim} \,(\tilde N(t,x_1)-\tilde N(t,x_2)).
\end{equation}
In view of \eqref{Nx} we compute
$$\tilde N(t,x_1)-\tilde N(t,x_2)=\int_{x_1}^{x_2}(-uT(t,x)+i\Phi_x N(t,x))e^{i\Phi}\,dx$$
$$=\int_{x_1}^{x_2}(-e^{iM\log\sqrt{t}}\sum_je^{-i|\alpha_j|^2\log \sqrt{t}}(\alpha_j+R_j(t))\frac{e^{i\frac{(x-j)^2}{4t}}}{\sqrt{t}}\,T(t,x)+i\Phi_x N(t,x))e^{i\Phi}\,dx.$$
The integral is well defined, and in view of the decay of $\{R_j(t)\}$ we have 
$$\tilde N(t,x_1)-\tilde N(t,x_2)=O(\sqrt{t})+\int_{x_1}^{x_2}(-e^{iM\log\sqrt{t}}\sum_je^{-i|\alpha_j|^2\log \sqrt{t}}\alpha_j\frac{e^{i\frac{(x-j)^2}{4t}}}{\sqrt{t}}\,T(t,x)+i\Phi_x N(t,x))e^{i\Phi}\,dx.$$
If we are in the case $x_1,x_2\in(n-1,n+1)$ and $\alpha_n=0$, the phase $x-j$ can vanish on $(x_1,x_2)$ only for $j=n$ but in this case the whole term vanishes as $\alpha_n=0$. In the case $(x_1,x_2)\in (n,n+1)$ the phase $x-j\neq 0$ cannot vanish on $(x_1,x_2)$.  
Therefore to get decay in time we integrate by parts:
$$\tilde N(t,x_1)-\tilde N(t,x_2)=O(\sqrt{t})+\left[-e^{iM\log\sqrt{t}}\sum_je^{-i|\alpha_j|^2\log \sqrt{t}}\alpha_j\frac{e^{i\frac{(x-j)^2}{4t}}}{\sqrt{t}}\frac{4t}{2i(x-j)}T(t,x)e^{i\Phi}\right]_{t_1}^{t_2}$$
$$+\int_{x_1}^{x_2}e^{iM\log\sqrt{t}}\sum_je^{-i|\alpha_j|^2\log \sqrt{t}}\alpha_j\frac{2\sqrt{t}}{i}e^{i\frac{(x-j)^2}{4t}}(\frac{1}{x-j}T(t,x)e^{i\Phi})_x+i\Phi_x N(t,x)e^{i\Phi}\,dx.$$
In view of  formula \eqref{Tx} for the derivative $T_x$ and the expression \eqref{phase} of $\Phi(t,x)$ we get
$$\tilde N(t,x_1)-\tilde N(t,x_2)=O(\sqrt{t})$$
$$+i\int_{x_1}^{x_2}(-e^{iM\log\sqrt{t}}2\sum_je^{-i|\alpha_j|^2\log \sqrt{t}}\alpha_je^{i\frac{(x-j)^2}{4t}}\frac{1}{x-j}$$
$$\times \Re(e^{-iM\log\sqrt{t}}\sum_je^{i|\alpha_k|^2\log \sqrt{t}}\overline{\alpha_k}e^{-i\frac{(x-k)^2}{4t}}N(t,x))+\Phi_x N(t,x))e^{i\Phi}\,dx$$
$$=O(\sqrt{t})+i\int_{x_1}^{x_2}(-\sum_{j,k}e^{-i(|\alpha_j|^2-|\alpha_k|^2)\log \sqrt{t}}\alpha_j\overline{\alpha_k}e^{i\frac{(x-j)^2-(x-k)^2}{4t}}\frac{1}{x-j}N(t,x)+\Phi_x N(t,x))e^{i\Phi}\,dx$$
$$-i\int_{x_1}^{x_2}e^{2iM\log\sqrt{t}}\sum_{j,k}e^{-i(|\alpha_j|^2+|\alpha_k|^2)\log \sqrt{t}}\alpha_j\alpha_ke^{i\frac{(x-j)^2+(x-k)^2}{4t}}\frac{1}{x-j}\overline{N(t,x)}e^{i\Phi}\,dx.$$
In the first integral the terms with $j=k$ cancel with the ones from $\Phi_x$. In the second integral the phase $(x-j)^2+(x-k)^2$ does not vanish as $(x_1,x_2)$ does not contain integers, so we can integrate by parts, use the expression \eqref{Nx} for $N_x$ and gain a $\sqrt{t}$ decay in time. We are left with
$$\tilde N(t,x_1)-\tilde N(t,x_2)=O(\sqrt{t})$$
$$-i\int_{x_1}^{x_2}\sum_{j\neq k}e^{-i(|\alpha_j|^2-|\alpha_k|^2)\log \sqrt{t}}\alpha_j\overline{\alpha_k}e^{i\frac{(x-j)^2-(x-k)^2}{4t}}\frac{1}{x-j}N(t,x)e^{i\Phi}\,dx.$$
If $n\pm \frac 12\notin(x_1,x_2)$ the phase $(x-j)^2-(x-k)^2$ does not vanish, so again we can perform an integration by parts to get the decay in time. If $n+\frac 12\in(x_1,x_2)\subset(n,n+1)$ we split the integral into three pieces : $(x_1,n+\frac 12-\sqrt{t}), (n+\frac 12-\sqrt{t},n+\frac 12+\sqrt{t})$ and $(n+\frac 12+\sqrt{t},x_2)$. On the middle segment, of size $2\sqrt{t}$ we upper-bound the integrant by a constant. On the extremal segments we perform an integration by parts, that gives a power $\sqrt{t}$ as 
$$\frac1{|(x-j)^2-(x-k)^2|}\leq \frac C{d(2x,\mathbb Z)}.$$ 
The cases when $n\pm\frac 12\in (x_1,x_2)\subset(n-1,n+1)$ are treated similarly. 
Therefore 
$$\tilde N(t,x_1)-\tilde N(t,x_2)=O(\sqrt{t}),$$
and in view of \eqref{convNdiff} we get the conclusion of the Lemma. 
\end{proof}

We end this section with a lemma that gives a link between $\tilde N(0,k^\pm)$ and the rotation $\Theta_k$ from Lemma \ref{lemmassstructure}.
\begin{lemma}\label{NRk}
Let $t_n$ be a sequence of positive times converging to zero, such that
\begin{equation}\label{condtn}e^{i\sum_j|\alpha_j|^2\log(\sqrt{t_n})}=1.\end{equation}
Using the notations in Lemma \ref{lemmassstructure} we have the following relation:
$$\Theta_k(B^\pm_{|\alpha_k|})=e^{-i\sum_{j\neq k}|\alpha_j|^2\log|k-j|}\,e^{i Arg(\alpha_k)}\,\tilde N(0,k^\pm).$$
\end{lemma}

\begin{proof}
Let $\epsilon>0$. We choose $x>0$ large enough such that in view of \eqref{ssstructurelim}
\begin{equation}\label{est1}|e^{i|\alpha_k|^2\log|x|}\,e^{i Arg(\alpha_k)}N_*(x)-\Theta_k(B^+_{|\alpha_k|})|\leq \frac\epsilon 4,\end{equation}
and such that 
\begin{equation}\label{est2}\frac{C(1+|k|)}{x}\leq\frac\epsilon 4,\end{equation}
where $C$ is the coefficient in \eqref{convN}.  
Then we choose $n$ large enough such that $|x\sqrt{t_n}|<\frac 12$ and such that 
\begin{equation}\label{est3}|e^{-i\sum_{j\neq k} |\alpha_j|^2\log |x\sqrt{t_n}+k-j|}-e^{-i\sum_{j\neq k} |\alpha_j|^2\log |k-j|}\leq\frac \epsilon 4,\end{equation}
and
\begin{equation}\label{est4}|e^{i|\alpha_k|^2\log \sqrt{t_n}}N(t_n,k+x\sqrt{t_n})-N_*(x)|\leq\frac \epsilon 4.\end{equation}
The last fact is possible in view of Lemma \ref{lemmassstructure} and \eqref{condtn}. Therefore we have, in view of Lemma \ref{lemmasegmN} :
$$I:=|e^{-i\sum_{j\neq k}|\alpha_j|^2\log|k-j|}\,e^{i Arg(\alpha_k)}\,\tilde N(0,k^+)-\Theta_k(B^+_{|\alpha_k|})|$$
$$=|e^{-i\sum_{j\neq k}|\alpha_j|^2\log|k-j|}\,e^{i Arg(\alpha_k)}\,\tilde N(0,k+x\sqrt{t_n})-\Theta_k(B^+_{|\alpha_k|})|$$
$$\leq |\tilde N(0,k+x\sqrt{t_n})-\tilde N(t_n,k+x\sqrt{t_n})|$$
$$+|e^{-i\sum_{j\neq k}|\alpha_j|^2\log|k-j|}\,e^{i Arg(\alpha_k)}\,\tilde N(t_n,k+x\sqrt{t_n})-\Theta_k(B^+_{|\alpha_k|})|.$$
By using the convergence \eqref{convN} of Lemma \ref{lemmaconvN} together with \eqref{est2}, and the definition \eqref{Ntilde} of $\tilde N$ we get 
$$I\leq\frac{\epsilon}4+ |e^{-i\sum_{j\neq k}|\alpha_j|^2\log|k-j|}\,e^{i Arg(\alpha_k)}\,e^{i\sum_j|\alpha_j|^2\log\frac{|x\sqrt{t_n}+k-j|}{\sqrt{t_n}}}\,N(t_n,k+x\sqrt{t_n})-\Theta_k(B^+_{|\alpha_k|})|.$$
In view of \eqref{est3} and \eqref{condtn} we have
$$I\leq \frac{2\epsilon}4+ |e^{i Arg(\alpha_k)}\,e^{i|\alpha_k|^2\log|x|}e^{-i\sum_{j\neq k}|\alpha_j|^2\log(\sqrt{t_n})}\,N(t_n,k+x\sqrt{t_n})-\Theta_k(B^+_{|\alpha_k|})|$$
$$=\frac{\epsilon}2+ |e^{i Arg(\alpha_k)}\,e^{i|\alpha_k|^2\log|x|}e^{i|\alpha_k|^2\log(\sqrt{t_n})}\,N(t_n,k+x\sqrt{t_n})-\Theta_k(B^+_{|\alpha_k|})|.$$
Finally, by \eqref{est4}
$$I\leq \frac{3\epsilon}4+  |e^{i Arg(\alpha_k)}\,e^{i|\alpha_k|^2\log|x|}\,N^*(x)-\Theta_k(B^+_{|\alpha_k|})|,$$
and we conclude by \eqref{est1} that
$$I\leq\epsilon,\forall\epsilon >0,$$
thus
$$\Theta_k(B^+_{|\alpha_k|})=e^{-i\sum_{j\neq k}|\alpha_j|^2\log|k-j|}\,e^{i Arg(\alpha_k)}\,\tilde N(0,k^+).$$
For $x<0$ we argue similarly to get
$$\Theta_k(B^-_{|\alpha_k|})=e^{-i\sum_{j\neq k}|\alpha_j|^2\log|k-j|}\,e^{i Arg(\alpha_k)}\,\tilde N(0,k^-).$$

\end{proof}

\subsection{Recovering the torsion of the initial data}\label{secttorsion}
Recall that in \S \ref{sectdesign} we have denoted by $\{x_n, n\in L\}$ the ordered set of the integer corner locations of $\chi_0$ and by $\{\theta_n,\tau_n,\delta_n\}_{n\in L}$ the sequence determining the curvature and torsion angles of $\chi_0$. 
Lemma \ref{lemmaangles} insured us that $\chi(0)$ has corners at the same locations as $\chi_0$, and of same angles. Let us denote $\{\theta_n,\tilde\tau_n,\tilde \delta_n\}_{n\in L}$ the correspondent sequence of $\chi(0)$. To recover $\chi_0$ up to rotation and translations we need to recover also the torsion properties of $\chi_0$, i.e. $\tilde\tau_n=\tau_n$ and $\tilde\delta_n=\delta_n$.

In \S \ref{sectdesign} we have defined the torsion parameters in terms of the vectorial product of two consecutive tangent vectors, and in view of the way the tangent vectors of $\chi(0)$ are described in Lemma \ref{lemmaangles}, we are lead to investigate vectorial products of type $\Theta_k(A^-_{|\alpha_k|}\wedge A^+_{|\alpha_k|})$. We start with the following lemma.
\begin{lemma}\label{Avect}
For $a>0$ there exists a unique $\phi_a\in[0,2\pi)$ such that
$$\frac{A^-_a\wedge A^+_a }{|A^-_a\wedge A^+_a |}=\Re(e^{i\phi_a}B^+_a)=-\Re(e^{-i\phi_a}B^-_a).$$

\end{lemma}
\begin{proof}
For simplicity we drop the subindex $a$. 
We recall from \eqref{asselfs} that the tangent vectors of the profile $\chi(1)$ have asymptotic directions the unitary vectors $A^\pm$ that can be described in view of formula (11) in \cite{GRV} as
$$A^+=(A_1,A_2,A_3),\quad A^-=(A_1,-A_2,-A_3).$$
This parity property for the tangent vector implies similar parity properties for normal and binormal vectors and from \eqref{asselfs} we also get
$$ B^+=(B_1,B_2,B_3),\quad B^-=(B_1,-B_2,-B_3).$$
In particular we have
$$ \frac{A^-\wedge A^+}{|A^-\wedge A^+ |}=\frac1{\sqrt{1-A_1^2}}(0,-A_3,A_2),$$
so
\begin{equation}\label{B+-} \frac{A^-\wedge A^+ }{|A^-\wedge A^+|}.B^+=\frac1{\sqrt{1-A_1^2}}(A_3B_2-A_2B_3)=-\frac{A^-\wedge A^+ }{|A^-\wedge A^+ |}.B^-.\end{equation}
Since $B^+\perp A^+$ and $\Re B^+,\Im B^+,A^+$ is an orthonormal basis of $\mathbb R^3$, we have a unique $\phi\in[0,2\pi)$ such that
$$\frac{A^-\wedge A^+}{|A^-\wedge A^+ |}=\cos\phi\,\Re B^++\sin\phi\,\Im B^+=\Re(e^{-i\phi}B^+),$$
thus the first inequality in the statement. The second inequality follows from \eqref{B+-}.
\end{proof}

We continue with some useful information on the connection between quantities involving normal components at two consecutive corners of $\chi(0)$. Recall that we have defined $\alpha_k=0$ if $k\notin\{x_n, n\in L\} $ and if $k=x_n$ for some $n\in L$ we have defined $\alpha_k\in\mathbb C$ by \eqref{moduluscoef} and \eqref{argcoef}. In particular two consecutive corners are located at $x_n$ and $x_{n+1}$, and the corresponding information is encoded by $\alpha_{x_n}$ and $\alpha_{x_{n+1}}$.  

\begin{lemma}\label{Rk} Let $t_n$ be a sequence of positive times converging to zero, such that the hypothesis \eqref{condtn} of Lemma \ref{NRk} holds. 
We have the following relation concerning two consecutive corners located at $x_n$ and $x_{n+1}$:
$$\Theta_{x_n}(B^+_{|\alpha_{x_n}|})=e^{i\beta_n}\,\,e^{i Arg(\alpha_{x_n})-Arg(\alpha_{x_{n+1}})}\Theta_{x_{n+1}}(B^-_{|\alpha_{x_{n+1}}|}),$$
where
$$\beta_n=(|\alpha_{x_n}|^2-|\alpha_{x_{n+1}}|^2)\log|x_n-x_{n+1}|.$$
\end{lemma}
\begin{proof}
The result is a simple consequence of Lemma \ref{NRk} and Lemma \ref{lemmasegmN}.
\end{proof}

Now we shall recover in the next lemma the modulus and the sign of the torsion angles of $\chi_0$. 
\begin{lemma}\label{lemmatorsion} 
The torsion angles of $\chi(0)$ and $\chi_0$ coincide:
$$\tilde\tau_n=\tau_n,\quad \tilde\delta_n=\delta_n,\quad \forall n\in L.$$
\end{lemma}
\begin{proof}Let $n\in L$ with $n+1\in L$ and $n\geq 0$ (the proof for $n\geq 0$ is similar).
From the definition \eqref{deftorsionmod} we have
$$\cos(\tilde\tau_{n+1})=\frac{T(0,x_{n}^-)\wedge T(0,x_{n}^+)}{|T(0,x_{n}^-)\wedge T(0,x_{n}^+)|}.\frac{T(0,x_{n+1}^-)\wedge T(0,x_{n+1}^+)}{|T(0,x_{n+1}^-)\wedge T(0,x_{n+1}^+)|}.$$
Now we use Lemma \ref{lemmaangles}:
$$\cos(\tilde\tau_{n+1})=\Theta_{x_n}\left(\frac{A^-_{|\alpha_{x_n}|}\wedge A^+_{|\alpha_{x_n}|}}{|A^-_{|\alpha_{x_n}|}\wedge A^+_{|\alpha_{x_n}|}|}\right).\Theta_{x_{n+1}}\left(\frac{A^-_{|\alpha_{x_{n+1}}|}\wedge A^+_{|\alpha_{x_{n+1}}|}}{|A^-_{|\alpha_{x_{n+1}}|}\wedge A^+_{|\alpha_{x_{n+1}}|}|}\right).$$
By using Lemma \ref{Avect} we write
$$\cos(\tilde\tau_{n+1})=\Theta_{x_n}\left(\Re(e^{i\phi_{|\alpha_{x_n}|}}B^+_{|\alpha_{x_n}|})\right).\Theta_{x_{n+1}}\left(-\Re(e^{i\phi_{|\alpha_{x_{n+1}}|}}B^-_{|\alpha_{x_{n+1}}|})\right)$$
$$=-\Re\left(\Theta_{x_n}(e^{i\phi_{|\alpha_{x_n}|}}B^+_{|\alpha_{x_n}|})\right).\Re\left(\Theta_{x_{n+1}}(e^{i\phi_{|\alpha_{x_{n+1}}|}}B^-_{|\alpha_{x_{n+1}}|})\right).$$
Finally, by Lemma \ref{Rk} we get
$$\cos(\tilde\tau_{n+1})=-\Re\left(e^{i\phi_{|\alpha_{x_n}|}+i\beta_n+i(Arg(\alpha_{x_n})-Arg(\alpha_{x_{n+1}}))}\Theta_{x_{n+1}}(B^-_{|\alpha_{x_{n+1}}|})\right)$$
$$.\Re\left(\Theta_{x_{n+1}}(e^{i\phi_{|\alpha_{x_{n+1}}|}}B^-_{|\alpha_{x_{n+1}}|})\right).$$
Since $\Re B^-_{|\alpha_{x_{n+1}}|}$ and $\Im B^-_{|\alpha_{x_{n+1}}|}$ are unitary orthogonal vectors, we obtain
$$\cos(\tilde\tau_{n+1})=-\cos(\phi_{|\alpha_{x_n}|}+\beta_n+Arg(\alpha_{x_n})-Arg(\alpha_{x_{n+1}})-\phi_{|\alpha_{x_{n+1}}|}).$$
Therefore, by definition \eqref{argcoef} of $\{Arg(\alpha_j)\}$ we get
$$\cos(\tilde\tau_{n+1})=\cos(\tau_{n+1}),$$
and in particular $\tilde\tau_{n+1}=\tau_{n+1}$.

Similarly, we compute
$$\frac{T(0,x_{n}^-)\wedge T(0,x_{n}^+)}{|T(0,x_{n}^-)\wedge T(0,x_{n}^+)|}\wedge\frac{T(0,x_{n+1}^-)\wedge T(0,x_{n+1}^+)}{|T(0,x_{n+1}^-)\wedge T(0,x_{n+1}^+)|}$$
$$=-\Theta_{x_{n+1}}(\Re B^-_{|\alpha_{x_{n+1}}|})\wedge \Theta_{x_{n+1}}(\Im B^-_{|\alpha_{x_{n+1}}|})\sin(\phi_{|\alpha_{x_n}|}+\beta_n+Arg(\alpha_{x_n})-Arg(\alpha_{x_{n+1}})-\phi_{|\alpha_{x_{n+1}}|}).$$
As $\Re(B^-_{|\alpha_{x_{n+1}}|})\wedge\Im(B^-_{|\alpha_{x_{n+1}}|})=A^-_{|\alpha_{x_{n+1}}|}$, in view of Lemma \ref{lemmaangles} we get
$$\frac{T(0,x_{n}^-)\wedge T(0,x_{n}^+)}{|T(0,x_{n}^-)\wedge T(0,x_{n}^+)|}\wedge\frac{T(0,x_{n+1}^-)\wedge T(0,x_{n+1}^+)}{|T(0,x_{n+1}^-)\wedge T(0,x_{n+1}^+)|}$$
$$=-T(0,x_{n+1}^-)\sin(\phi_{|\alpha_{x_n}|}+\beta_n+Arg(\alpha_{x_n})-Arg(\alpha_{x_{n+1}})-\phi_{|\alpha_{x_{n+1}}|}),$$
so by definition \eqref{argcoef} of $\{Arg(\alpha_j)\}$  we conclude $\tilde\delta_{n+1}=\delta_{n+1}$.
\end{proof}

\subsection{End of the existence result proof}\label{sectendproof} From Lemma \ref{lemmaangles} and Lemma \ref{lemmatorsion} we conclude that $\chi(0)$ and $\chi_0$ have the same characterizing sequences $\{\theta_n,\tau_n,\delta_n\}_{n\in L}$. In view of the definition of this sequence in \S\ref{sectdesign} we conclude that $\chi(0)$ and $\chi_0$ coincide modulo a rotation and a translation. This rotation and translation can be removed by changing the initial point $P$ and frame $(v_1,v_2,v_3)$ used in the construction of $\chi(t)$ in \S\ref{sectconstr}. Therefore we have constructed the curve evolution in Theorem \ref{brokenline} for positive times. 
The extension in time is done by using the time reversibility of the Schr\"odinger equation and the one of the binormal flow, that means solving  for positive times the binormal flow with initial data $\chi(-s)$, which is still a polygonal line satisfying the hypothesis.

\subsection{Further properties of the constructed solution}\label{sectproperties}  In this last subsection we describe the trajectories in time of the $\mathbb R^3-$locations of the corners, $\chi(t,k)$. 

\begin{lemma}
Let $k$ such that $\alpha_k\neq 0$, that is a location of corner for $\chi_0$. Then there exists two orthogonal vectors $v_1,v_2\in\mathbb S^2$ such that
$$\chi(t,k)=\chi(0,k)+\sqrt{t}\,(v_1+iv_2)+O(t).$$
\end{lemma}

\begin{proof}
From \eqref{chit} and the decay of $\{R_j(\tau)\}$ we have
$$\chi(t,k)-\chi(0,k)=\int_0^t\Im(\overline{u}N(\tau,k))\,d\tau$$
$$=\Im\int_{0}^{t}e^{-iM\log\sqrt{\tau}}\sum_je^{i|\alpha_j|^2\log \sqrt{\tau}}(\overline{\alpha_j+R_j(\tau)})\frac{e^{-i\frac{(k-j)^2}{4\tau}}}{\sqrt{\tau}}\,N(\tau,k)\,d\tau$$
$$=\Im\int_{0}^{t}e^{-iM\log\sqrt{\tau}}\sum_je^{i|\alpha_j|^2\log \sqrt{\tau}}\overline{\alpha_j}\frac{e^{-i\frac{(k-j)^2}{4\tau}}}{\sqrt{\tau}}\,N(\tau,k)\,d\tau+O(t^{\frac 12+\gamma}).$$
In the terms with $j\neq k$ we perform an integration by parts to get decay in time
$$\chi(t,k)-\chi(0,k)=\Im \int_{0}^{t}e^{-iM\log\sqrt{\tau}}e^{i|\alpha_k|^2\log \sqrt{\tau}}\overline{\alpha_k}\,N(\tau,k)\,\frac{d\tau}{\sqrt{\tau}}+O(t^{\frac 12+\gamma})$$
$$+\left[\Im e^{-iM\log\sqrt{\tau}}\sum_{j\neq k} e^{i|\alpha_j|^2\log \sqrt{\tau}}\overline{\alpha_j}\frac{e^{-i\frac{(k-j)^2}{4\tau}}}{\sqrt{\tau}} \frac{4\tau^2}{i(k-j)^2}\,N(\tau,k) \right]_0^t$$
$$-\Im \int_{0}^{t} e^{-iM\log\sqrt{\tau}}\sum_{j\neq k} 4\overline{\alpha_j}\frac{e^{-i\frac{(k-j)^2}{4\tau}}}{ i(k-j)^2}(e^{i|\alpha_j|^2\log \sqrt{\tau}}\tau\sqrt{\tau}\,N(\tau,k))_\tau \,d\tau.$$
The boundary term is of order $O(t\sqrt{t})$. In view of \eqref{Nt} we get a $\frac 1{\tau\sqrt{\tau}}$ estimate for $N_\tau$ so the last term is of order $O(t)$, and we have
$$\chi(t,k)-\chi(0,k)=\Im \int_{0}^{t}e^{-iM\log\sqrt{\tau}}e^{i|\alpha_k|^2\log \sqrt{\tau}}\overline{\alpha_k}\,N(\tau,k)\,\frac{d\tau}{\sqrt{\tau}}+O(t).$$
Now, from \eqref{03} in Lemma \ref{lemmaconvN} we get the existence of $w_1,w_2\in\mathbb S^2$ such that 
$$w_1+iw_2=\underset{t\rightarrow 0}{\lim}\,e^{-i\sum_{j\neq k}|\alpha_j|^2\log\sqrt t}N(t,k),$$
with a rate of convergence upper-bounded by $C(1+|k|)\sqrt{t}$. This implies
$$\chi(t,k)-\chi(0,k)=\Im\, \overline{\alpha_k}\,(w_1+iw_2)\int_{0}^{t}e^{-iM\log\sqrt{\tau}}e^{i\sum_j|\alpha_j|^2\log \sqrt{\tau}}\,\frac{d\tau}{\sqrt{\tau}}+O(t)$$
$$=\Im\, \overline{\alpha_k}\,(w_1+iw_2)\int_{0}^{t}\,\frac{d\tau}{\sqrt{\tau}}+O(t)$$
and thus get the conclusion of the Lemma.

\end{proof}

\end{document}